\documentclass[oneside,english]{amsart}
\usepackage[T1]{fontenc}
\usepackage[latin9]{inputenc}
\usepackage{amstext}
\usepackage{amsthm}
\usepackage{amssymb}
\usepackage{babel}

\usepackage{etoolbox}
\patchcmd{\part}{\normalfont}{\normalfont\Large}{}{}

\makeatletter
\numberwithin{equation}{section}
\numberwithin{figure}{section}
\theoremstyle{plain}
\newtheorem{thm}{Theorem}[section]

\newtheorem*{thm*}{Theorem}

\theoremstyle{definition}
\newtheorem{defn}[thm]{\protect\definitionname}
\theoremstyle{plain}
\newtheorem{lem}[thm]{\protect\lemmaname}
\theoremstyle{remark}
\newtheorem{rem}[thm]{\protect\remarkname}
\theoremstyle{remark}
\newtheorem{claim}[thm]{\protect\claimname}
\theoremstyle{remark}
\newtheorem{ex}[thm]{\protect\examplename}

\makeatother
\newcommand{\remove}[1]{}
\usepackage{babel}
\providecommand{\claimname}{Claim}
\providecommand{\definitionname}{Definition}
\providecommand{\lemmaname}{Lemma}
\providecommand{\remarkname}{Remark}
\providecommand{\examplename}{Example}

\newcommand{\mc}[1]{\mathcal{#1}}
\newcommand{\mb}[1]{\mathbb{#1}}

\newcommand{\sub}{\subseteq}

\newcommand{\sm}{\setminus}
\newcommand{\ov}{\overline}

\newcommand{\eps}{\varepsilon}
\newcommand{\es}{\emptyset}

\newcommand{\ua}{\uparrow}

\newcommand{\aA}{\alpha}
\newcommand{\bB}{\beta}
\newcommand{\gG}{\gamma}

\newcommand{\kK}{\kappa}

\newcommand{\lL}{\lambda}
\newcommand{\tT}{\theta}
\newcommand{\sS}{\sigma}

\newcommand{\OO}{\Omega}
\newcommand{\TT}{\Theta}

\begin{document}
\title[Sharp hypercontractivity in symmetric groups]{Sharp hypercontractivity for symmetric groups and its applications}

\date{}

\author{Peter Keevash and Noam Lifshitz}

\address{Peter Keevash, Mathematical Institute, University of Oxford, UK.
Supported by ERC Advanced Grant 883810. \newline
Noam Lifshitz, Einstein Institute of Mathematics, Hebrew University, Jerusalem, Israel. Supported by the Israel Science Foundation (grant no.~1980/22).}

\begin{abstract}
A recently fertile strand of research in Group Theory is developing
non-abelian analogues of classical combinatorial results for arithmetic Cayley graphs,
describing properties such as growth, expansion, mixing, diameter, etc.
We consider these problems for the symmetric and alternating groups.
The case of normal Cayley graphs (those generated by unions of conjugacy classes)
has seen significant progress via character theory
(whereby Larsen and Shalev resolved several open problems),
but the general case still remains poorly understood.
In this paper we generalise the background assumption from being normal 
to being global (a pseudorandomness condition), replacing character bounds
by spectral estimates for convolution operators of global functions,
thus obtaining qualitative generalisations of several results on normal Cayley graphs.
Furthermore, our theory in the pseudorandom setting can be applied 
(via density increment arguments) to several results for general sets 
that are not too sparse, including analogues of 
Polynomial Freiman-Ruzsa,  Bogolyubov's lemma, Roth's theorem,  the Waring problem
and essentially sharp estimates for the diameter problem of Cayley graphs whose density is at least exponential in $-n$. Our main tool is a sharp new hypercontractive inequality for global functions on the symmetric group.   
\end{abstract}

\maketitle
\section{Introduction}

The diameter of a Cayley graph $\mathrm{Cay}(G,A)$ with $A=A^{-1}$ is the smallest $m$ for which $A^m=G$,
also known as the  \emph{covering number} $\mathrm{cn}(A)$. A prominent conjecture of Babai (see \cite{babai1992diameter})
states that the diameter of any finite connected $\mathrm{Cay}(G,A)$ is at most polylogarithmic in $|G|$.
This conjecture remains widely open, even when $G$ is the symmetric group $S_n$ or alternating group $A_n$,
for which the best known bound, due to Helfgott and Seress \cite{helfgott2014diameter}, is quasipolynomial in $\log |G|$.
The apparent difficulty of the problem has naturally led to interest in special cases, particularly that 
of \emph{normal} Cayley graphs, i.e.\ those where $A$ is normal, or even just a single conjugacy class. It turns out that even these special cases are highly non-trivial and have connections to prominent open problems in group theory related to word maps.   
A strong form of Babai's Conjecture for normal sets follows from a result of Liebeck and Shalev \cite{liebeck2001diameters},
and Larsen and Shalev~\cite{larsen2008characters} gave very precise estimates for $\mathrm{cn}(A)$
when $G=A_n$ and $A$ is a conjugacy class with not too many fixed points,
in particular resolving a conjecture of Rudvalis and Vishne.
The key tool developed by Larsen and Shalev was an improved bound for character values,
by which means they obtained several other striking results on conjugacy class Cayley graphs,
including a solution to a mixing time conjecture of Lulov and Pak, 
and a word map analogue of Waring's Problem.
More generally, these results may be considered part of a general program (largely open)
of extending a variety of classical problems on Cayley graphs to finite simple groups.
Many of these problems are discussed in the surveys 
\cite{breuillard2018expansion, helfgott2015growth, shalev2013some, tao2015expansion}.

Philosophically, the character theoretic approach is a manifestation of pseudorandomness,
which suggests analogies with the corresponding theory in the integers,
where there is a long and rich history of deep theorems 
on the additive pseudorandomness properties of multiplicative sets,
obtained via Fourier transforms and the Hardy-Littlewood circle method.
In the non-abelian setting, for a Cayley graph generated by the conjugacy class of $\sS$,
the character ratios $\chi(\sS)/\chi(1)$ are the eigenvalues of the Cayley adjacency operator,
so estimates thereof naturally apply to questions of expansion, mixing and growth.
For a general Cayley graph, one can still exploit some tools from representation theory,
given that general adjacency operators preserve the isotypic components,
but these components are no longer eigenspaces, so one can only hope that 
pseudorandomness is manifested through bounds on the spectral norm.
We will indeed establish such bounds, via a new hypercontractive inequality for $S_n$
that provides a sharp variant of a previous result  
of Filmus, Kindler, Lifshitz, and Minzer~\cite{filmus2020hypercontractivity}.
We will apply these bounds to extend several results on conjugacy class Cayley graphs
to the pseudorandom setting, and also prove several results for general sets
that are not too sparse, including analogues of
Polynomial Freiman-Ruzsa,  Bogolyubov's lemma, Roth's theorem, the Waring problem
and essentially sharp estimates for the diameter problem.

\subsection{Results for general Cayley graphs}

Now we will state our results on Cayley graphs of $A_n$
for general generating sets $A$ that are not too sparse.
We adopt the uniform measure $\mu(A) = |A|/|A_n|$.

\subsubsection*{Diameter}
The following bound on the covering number
applies to any set of density $e^{-O(n)}$ and is sharp up to a constant factor.

\begin{thm}\label{thm:diameter bound}
There exist absolute constants $c,C>0$ so that if $A\subseteq A_n$ 
with $A=A^{-1}$, $\mu(A) \ge e^{-cn^{1-\eps}}$ and $\eps>1/\log n$
then $\mathrm{cn}(A) \le 8\eps^{-1} + C\eps^{-2} \log_n^2(1/\mu(A))$. 
\end{thm}

To see sharpness, we consider $A$ 
obtained from $\bigcap_{i=1}^n \{ \sS \in A_n: \sS(i)=i \}$
by adding all $3$-cycles on consecutive numbers.
Clearly $\log_n(1/\mu(A)) \sim m$ for $m=o(n)$.
Furthermore, $A$ has covering number $\OO(m^2)$
by considering the cycle $(m,m-1,\ldots ,1)$, noting that
when multiplying elements of $A$ each new term
can add only $O(1)$ inversions $(i,j)$ with $i,j \le m$.

\subsubsection*{Polynomial Freiman-Ruzsa}
A subset $A$ of a group $G$ is called a $K$-approximate group 
if $\text{id} \in A =A^{-1}$  and $A^2 \subseteq XA$ for some $X \sub G$ with $|X|=K$.
The qualitative structure of approximate groups is quite well understood,
in the abelian setting by Freiman's Theorem \cite{freiman1999structure} 
(extended to abelian groups by Green and Ruzsa \cite{green2007freiman}),
and in general by Breuillard, Green and Tao \cite{breuillard2012structure}
(a qualitative answer to a conjecture of Helfgott and Lindenstrauss).
However, the quantitative aspects of this structure are poorly understood,
except that for finite groups of Lie type with bounded rank the conjecture
was proved by Eberhard, Murphy, Pyber and Szab\'o \cite{eberhard2021growth}
(see also \cite{pyber2016growth}).
The abelian case has considerable interest in its own right
and from the perspective of diverse applications
(ranging to Computer Science, see \cite{lovett2015exposition});
here the influential Polynomial Freiman-Ruzsa Conjecture (see \cite{PFR}) 
suggests that approximate groups can be covered by polynomially few translates
of a canonical approximate group (a convex progression of the correct size and dimension).

We establish such a covering result for $K$-approximate subgroups 
$A$ in $S_n$ with $\mu(A) = e^{-O(n^{1-\eps})}$:
they can be efficiently covered by cosets of a large subgroup, 
with parameters depending polynomially on $\mu(A)$ and $K$.

\begin{thm}\label{thm:approximate subgroups_intro}
There is $c>0$ so that if $A\sub S_n$ is a $K$-approximate group
with $\mu(A) \ge e^{-cn^{1-\eps}}$ and $\eps>1/\log n$
then there is a subgroup $H$ with $\mu(H) \ge \mu(A)^{5/\eps}$
such that $A \sub YH$ for some $Y \sub A$ with $|Y| \le K^6 |A|/|H|$.
\end{thm}

\subsubsection*{Bogolyubov}
A key ingredient of Ruzsa's approach \cite{ruzsa1994generalized}
to Freiman-type theorems, known as Bogolyubov's Lemma, 
is finding structure in iterated sumsets 
(Bohr sets in $\mb{Z}$, or subspaces in $\mb{F}_q^n$).
Such a result in $S_n$ was an important ingredient for Helfgott and Seress
in their diameter result mentioned above: they showed
(see  \cite[Proposition 3.15]{helfgott2014diameter})
that if $A \sub A_n$ with $A=A^{-1}$, $\mu(A) \ge d^n$ and $d>1/2$
then $A^{8n^5}$ contains a large subgroup,
indeed a naturally embedded $A_{dn}$
of the form $U_I := \bigcap_{x \in I} \{ \sS \in A_n: \sS(x)=x\}$.
We obtain the following quantitative improvement of their result,
reducing the exponent of $A$ to a constant and moreover finding 
a polynomially large subgroup (which is much stronger than
the bounds in the usual arithmetic Bogolyubov).

\begin{thm}\label{thm:Bogolyubov_intro}
There exists $c>0$ so that if $A \sub A_n$ 
with $A=A^{-1}$,  $\mu(A) \ge e^{-cn^{1-1/M}}$ 
and $M < \log n$ is an integer
then $A^{8M}$ contains a subgroup $U_I$
with $\mu(U_I) \ge \mu(A)^{5M}$.
\end{thm}

\subsubsection*{Waring}

The Waring problem, that became a theorem of Hilbert in 1909,
states that for any $k \in \mb{N}$ there is some $t \in \mb{N}$
such that any $n \in \mb{N}$ is a sum of at most $t$ 
natural numbers raised to their $k$th powers.
An analogous result of Shalev \cite{shalev2009word} 
in the non-abelian setting shows that for any sufficiently
large finite simple group $G$ every element of $G$
can be expressed as a product of three $k$th powers. In fact, Shalev's result applies not only to $k$th powers,
but also to the image $w(G)$ of any word map,
where for any non-trivial finite word $w$ in a free group $F$
we write $w(G)$ for the set of all possible elements of $G$
obtained by substituting elements of $G$ for the generators of $F$. 
Nikolov and Pyber \cite{np} proved a robust version of this result,
replacing $w(G)$ by a sufficiently dense subset of any set
of the form $w_1(G) \cap \dots \cap w_k(G)$, 
but their proof does not apply to $A_n$. We fill this gap. 

\begin{thm} \label{thm:waring}
For any non-trivial finite words $w_1,\dots,w_k$
there is $N \in \mb{N}$ so that if $n \ge N$
and $A$ has density at least $n^{-1/4}$
in $w_1(A_n) \cap \dots \cap w_k(A_n)$,
then $A^3=A_n$.
\end{thm}

Larsen and Shalev \cite{larsen2008characters} and then Larsen, Shalev, and Tiep \cite{larsen2009word} improved on Shalev~\cite{shalev2009word} by showing that every element of a sufficiently large finite simple group $G$ is a product of two elements of $w(G)$. However, the robust analogue completely fails for their version. Indeed, one can easily find a large subset  $A\subseteq w(G)$ with $1\notin A^2$. 

\subsubsection*{Roth}
A classical theorem of Roth \cite{roth1953certain},
that inspired many later developments in Additive Combinatorics (see \cite{tao2006additive}),
shows that any positive density subset of $\mb{Z}$ contains a non-trivial three-term arithmetic progression (3AP).
There is a rich literature on the quantitative aspects of this question,
recently brought to a dramatic conclusion by Kelley and Meka \cite{kelley2023strong},
who showed that one can find a 3AP in any subset of $\{1,\dots,n\}$ 
of a certain quasi-polynomial density, thus matching a classical construction 
of Behrend \cite{behrend1946sets} that gives a lower bound of the same form.
For general groups, a qualitative analogue of Roth's theorem (and solutions of more general equations)
was established by Kr\'al', Serra and Vena \cite{kral2009}. Their `regularity type' bounds for the maximum density 
were improved by Sanders \cite{sanders2017solving} to doubly logarithmic bounds for the analogue of Roth's theorem,
where in a general group we say $(x,y,z)$ is a 3AP if $xz=y^2$, which is non-trivial if $x,y,z$ are not all equal.
We obtain the following improvement of the density upper bound in $S_n$  from $\log^{-c} n$ to $n^{-c\log n}$.

\begin{thm}\label{thm: Our variant of roth_intro}
There is an absolute constant $c>0$ so that 
any $A \sub S_n$ with no 3AP has $\mu(A) < n^{-c\log n}$.
\end{thm}

\remove{
\begin{itemize}
\item character ratios (good for growth of conjugacy classes)
\item Various motivations for studying non-class functions
\item The quantity generalizing the character ratios for non class functions.
\item Approximate subgroups
\item Bogolyubov 
\item Roth for $S_n$
\item Diameter bounds
\item Quasirandomness and its central role in constructing expanders
\item The global mixing property 
\item mixing times for global sets
\item Future application: Character bounds in the symmetric group
\item Via hypercontractivity for global functions
\end{itemize}
}

\subsection{Globalness and Analysis in symmetric groups}

Next we will state our new results for Analysis in symmetric groups,
i.e.\ our spectral and hypercontractive inequalities 
under the pseudorandomness notion of globalness.

\subsubsection*{Globalness} 
We start with some basic notation and the definition of globalness.
We consider the space $L^2(S_n,\mu)$ of real-valued functions on $S_n$
(also identified with the group algebra $\mb{R}S_n$) under the uniform measure $\mu$ on $S_n$.
We write $L^2$ to emphasise the role of the inner product
$\langle f, g\rangle =\mb{E}_{\sS \in S_n} [f(\sigma)g(\sigma)]$,
but we will also consider other $L^p$ norms below.

Let $[n]_t$ be the set of $t$-tuples of distinct element of $[n]$.
For $I,J \in [n]_t$ we write $U_{I \to J} = \{ \sS \in S_n: \sS \circ I = J\}$ and $U_I = U_{I \to I}$.
Following Friedgut \cite{friedgut2008measure}, we call $U_{I\to J}$ a \emph{$t$-umvirate}.
Equivalently, $t$-umvirates are cosets of pointwise stabilisers of $t$-sets.
We write $f_{I\to J} \in L^2(U_{I \to J},\mu)$ for the restriction of $f$ to $U_{I\to J}$,
where $\mu$ also denotes uniform measure on $U_{I\to J}$.

\begin{defn} \label{def:global}
We say that $f$ is $(r,\gG)$-\emph{global} if for each $d$ and $I,J \in [n]_d$ 
we have $\|f_{I\to J}\|_{2} \le r^d \gG$.
We say  that $f$ is $r$-\emph{global} if it is $(r,\|f\|_2)$-global. 
We say that $A\subseteq S_n$ is $r$-global if its indicator function $1_A$ is $r$-global.  
\end{defn}

Roughly speaking, a function $f$ is global if it is not highly correlated with any $t$-umvirate;
for example, if $f=1_A$ is $r$-global then restricting to any $t$-umvirate
increases the density by a factor at most $r^{2t}$.
The philosophy of the results of our paper is that global sets 
exhibit similar growth properties to the conjugacy classes. 
Intuitively, this means that the $t$-umvirates are the only obstruction to growth 
inside large subsets of the symmetric group.    

An instructive example is the subgroup $A = S_{n/2}\times S_{n/2}$,
which is $4$-global, has density about $2^{-n}$, 
yet $\text{Cay}(A_n,A)$ is disconnected.
This indicates the difficulty of extending our results
to densities smaller than $2^{-n}$.

\subsubsection*{The level $d$ inequality}

To motivate our next idea we start by recalling its analogue in the Analysis of Boolean functions. 
Any function $f\colon\{-1,1\}^n\to \mathbb{R}$ has a Fourier expansion $f=\sum \hat{f}(S)\chi_S$, 
where $\chi_S(x)=\prod_{i\in S}x_i.$ The coarser orthogonal decomposition
$f=\sum_{d=0}^n f^{=d}$, where $f^{=d}:=\sum_{|S|=d}\hat{f}(S)\chi_S$,
is known as the {\em degree decomposition}, 
as each function $f^{=d}$ is a polynomial of degree $d$. 
 
 A fundamental result known as the \emph{level $d$ inequality} states that 
 for $A\subseteq \{-1,1\}^n$ of density $\frac{|A|}{2^n}=\aA$, 
 if $d\le 2\log(1/\aA)$ then  $\|1_A^{=d}\|_2^2\le \aA^2 (2ed^{-1}\log(1/\aA))^d$.
 (See \cite[Section 9.5]{o2014analysis} or the original papers of Kahn, Kalai, and Linial \cite{kahn1988influence} 
 and Benjamini, Kalai, and Schramm  \cite{benjamini1999noise}.)
 To interpret this result, note that $\|1_A\|_2^2 = \aA = \sum_d \|1_A^{=d}\|_2^2$ 
 so for small $\aA$ we see that the low degrees have low Fourier weight.
 
Following Ellis, Friedgut and Pilpel~\cite{ellis2011intersecting}, we introduce the corresponding degree decomposition for the symmetric group,
which makes an analogy between the dictators $f(x)=x_i$ on $\{-1,1\}^n$ 
and the $1$-umvirates (also called dictators) $x_{i \to j} := 1_{U_{i\to j}}$ on $S_n$.

\begin{defn} \label{def:degdec}
Let $V_{\le d} \le L^2(S_n)$ consider of all functions $f(\sS) = g((x_{i\to j}(\sigma))_{i,j})$,
where $g$ is a multivariate degree $d$ polynomial in $n^2$ variables $(x_{ij})$.
The \emph{degree} of $f \in  L^2(S_n)$ is the smallest $d$ such that $f\in V_{\le d}$.
We write $f^{\le d}$ for the orthogonal projection of $f$ on $V_{\le d}$.
The \emph{degree $d$ part} of $f$ is $f^{=d} := f^{\le d} - f^{\le d-1}$. 
\end{defn} 

Now we can state a simplified form of our level $d$ inequality for global functions of the symmetric group. 

\begin{thm}\label{thm:level-d for global functions_intro}
For some absolute constant $C>0$, 
if $A \sub S_n$ is $r$-global and $d\le \min(\tfrac{1}{8}\log(1/\mu(A)), 10^{-5}n)$ then 
$\|1_A^{=d} \|_2^2\le \mu(A)^2 \left( C r^4 d^{-1} \log (1/\mu(A)) \right)^d$.
\end{thm}

A more refined version (see Theorem \ref{thm:level-d for biglobal functions} below),
which applies to `biglobal' general functions (not just sets), will be the key tool from Analysis
underpinning the applications listed in the previous section.
For comparison with the level $d$ inequality of Filmus, Kindler, Lifshitz and Minzer~\cite{filmus2020hypercontractivity},
the key point to note is the optimal power of $\log (1/\mu(A))$, which makes our inequality effective
for $\mu(A) > \exp -O(n)$, as opposed to $\mu(A) > \exp -O(n^c)$ for some $c>0$;
this is sharp (see Example \ref{ex:leveld}).
We achieve this improvement, which is the main new technical contribution of the paper,
by combining a variety of ideas from Representation Theory, Analysis and Algebraic Combinatorics.
 
\subsubsection*{Spectrum of global convolution operators}
 
As mentioned above, it is a major goal of character theory to estimate
the character ratios $\chi(\sS)/\chi(1)$, which can be interpreted as eigenvalues
of the adjacency operator $T_\sS$ corresponding to the random walk on $\mathrm{Cay}(S_n,\sigma^{S_n})$,
defined by $(T_\sS f)(\pi) = \mb{E}_{\tau \sim \sigma^{S_n}}[f(\tau \pi)]$. We will generalise $T_\sS$
to the convolution operator $T_f$ defined by $T_f(g) = f*g$, where for any $f,g \in L^2(S_n)$
we define their convolution by 
\[f*g(\sigma) = \mathbb{E}_{\tau \sim S_n}[f(\tau)g(\tau^{-1}\sigma)].\]
It is not hard to see that $T_f$ preserves each level $V_{=d}$;
in fact, it commutes with the right action of $S_n$, so preserves each isotypic component
of the representation theoretic decomposition (to be discussed below).
In the general setting, we replace character bounds by the following estimates on
\[ \| T_f \|_{V^{=d}} := \sup \{ \|T_f g\|_2: g \in V^{=d}, \|g\|_2=1 \}, \]
which is the $L^2 \to L^2$ norm of $T_f$ on $V_{=d}$.
(If $f(\sS)=f(\sS^{-1})$ is symmetric then $T_f$ is self-adjoint, 
so $\| T_f \|_{V^{=d}}$ is the spectral radius of $T_f$ restricted to $V^{=d}$.)
Again, we state a simplified form of the result
specialised to functions of the form $1_A/\mu(A)$.

\begin{thm}\label{thm:Eigenvalues of operators_intro} 
For some absolute constant $C>0$, 
if $f=1_A/\mu(A)$ for some $r$-global $A \sub S_n$
and $d\le \min\left(\frac{1}{8}\log(1/\mu(A)), n/10^5\right)$
then \[ \|T_f\|_{V_{=d}} \le \left( Cr^4 n^{-1} \log(1/\mu(A)) \right)^{d/2}.\] 
\end{thm}

\subsubsection*{Hypercontractivity}
The main tool in the proof of the level $d$ inequality for functions on $\{-1,1\}^n$
is the classical hypercontractivity theorem of Bonami  (see \cite[Chapter 9]{o2014analysis}),
which has been hugely influential in the Analysis of Boolean functions
and its broader applications in Mathematics and Computer Science.
This states that the \emph{noise operator} $\mathrm{T}_{\rho}$ given by
$\mathrm{T}_{\rho} f = \sum_{d=0}^n \rho^{d}f^{=d}$ is a contraction 
as an operator from $L_p$ to $L_q$ for all $q>p\ge 1$ such that  $\rho \le \sqrt{\frac{p-1}{q-1}}$. 
We give the following analogue for global functions on the symmetric group. 

\begin{thm}
\label{thm:hypercontractivity for global functions in the symmetric group}
There is a family of self-adjoint operators $\{\mathrm{T}_{\rho}: \rho \in (0,1) \}$
on $L^2(S_n)$ with the following properties.
\begin{enumerate}
\item (Large degree $d$ eigenvalues) 
If $f \in V_{\le d}$ with $d \le 10^{-5} n$ then
\[
\left\langle \mathrm{T}_{\rho}f,f\right\rangle \ge (\rho/72)^d \|f\|_{2}^{2}.
\] 
\item If $f$ is $\left(r,\gamma\right)$-global, $q \ge 2$ and $\rho\le\frac{\log q}{16rq}$ then 
\[
\|\mathrm{T}_{\rho}f\|_{q}^{q}\le\gamma^{q-2}\|f\|_{2}^{2}.
\]
\item Each $\mathrm{T}_{\rho}$ commutes with the action of $S_{n}$ from both sides. 
\end{enumerate}
\end{thm}

As for the level $d$ inequality, this gives essentially optimal estimates
and improves the previous hypercontractive inequality on $S_n$
by Filmus, Kindler, Lifshitz and Minzer~\cite{filmus2020hypercontractivity}.
While the main motivation of our paper is to extend the results of Larsen and Shalev 
from conjugacy classes to arbitrary sets, in a future paper Lifshitz and Marmor~\cite{lifshitz2023bounds} 
apply Theorem \ref{thm:hypercontractivity for global functions in the symmetric group} 
to give essentially sharp estimates of the $q$-norms of the characters of the symmetric group,
thus improving various results of Larsen and Shalev on mixing times of conjugacy classes. 

\subsection{Mixing}

Now we consider the mixing properties of the random walk corresponding to a global set $A \subset A_n$.
We consider the convolution operator $T_f$, where $f = 1_A/\mu(A)$.
The $(L_p,\epsilon)$-mixing time of the random walk in $\mathrm{Cay}(A_n,A)$ 
is the minimal $M$ for which $\|f^{*M}-1\|_p<\epsilon$,
where $f^{*M} = f * \cdots * f$ is the $M$-fold convolution 
and $1$ denotes the function identically equal to $1$.
To see that this is equivalent to the usual definition in terms of probability distributions,
we identify any non-negative $f$ with $\|f\|_1=1$ with the distribution $\nu_f(\sS) = |A_n|^{-1} f(\sS)$;
then \[ \nu_{f*g}(\sS) = |A_n|^{-1} (f*g)(\sS) = |A_n|^{-2} \sum_{\tau} f(\tau)g(\tau^{-1} \sS) 
= \mb{P}_{\tau \sim \nu_f, \tau' \sim \nu_g} \{ \tau \tau' = \sS \}.\] 
Much of the theory of mixing times (see \cite{wilmer2009markov} for an introduction)
concerns the $(L_1,1/4)$-mixing time (known simply as the mixing time)
and the relaxation time (the reciprocal of the spectral gap, equivalent to $L_2$-mixing).
Mixing times for other $L_p$ norms can be controlled up to a constant factor by $L_2$-mixing,
noting that $\| g * h \|_\infty \le \|g\|_2 \|h\|_2$ by Young's convolution inequality.

Larsen and Shalev~\cite{larsen2008characters} obtained several mixing results for $T_f$ as above
when $A \sub A_n$ is a conjugacy class; for example, when $\mu(A) > e^{-n^{1-1/M+\eps}}$ 
the $(L_2,n^{-\eps})$-mixing time is at most $M$  (see \cite[Theorem 6.1]{larsen2008characters}).
We obtain a similar result in the more general context of biglobal functions,
which is sharp for both $L_2$-mixing and $L_1$-mixing 
(see Section \ref{sec:Sharpness}).

\begin{thm}\label{thm:mixing time_intro}
For each $\eps>0$ there exists $c>0$ so that
if $M<\log n$ is an integer and $f\in L^2(A_n)$ is $r$-biglobal
with $\|f\|_1=1$ and $\|f\|_2 \le e^{cr^{-2}n^{1-1/M}}$,
then $\|f^{*M}-1\|_2 < \eps$. 
\end{thm}

Next we consider mixing properties related to product-free sets.
The largest product-free sets in $A_n$ were recently determined by  
Keevash, Lifshitz, and Minzer \cite{keevash2022largest}
(building on progress by Gowers \cite{gowers} and Eberhard \cite{eberhard}),
thus answering a question of Babai and S\'{o}s \cite{bs} from 1985.
For large $n$, the largest product-free sets in $A_n$ 
(proposed by Kedlaya \cite{kedlaya1}) take the form
 \[K_{x,I}:=\{\sigma\in A_n:\,\sigma(x)\in I, \sigma(I)\cap I = \es \}.\] 
Moreover, there is a stability result (see \cite[Theorem 1.5]{keevash2022largest})
that if $A \sub A_n$ is product-free with $\mu(A) > n^{-r}$ then there is some
$t$-umvirate with $t \le 4r$ in which $A$ has density at least $n^{t/4}\mu(A)$.
In contrapositive form, if $A$ is product-free and sufficiently global
then for large $n$ its density decays faster than any polynomial in $n$.
Here we will obtain the following much stronger estimate for the density. 

\begin{thm}\label{cor:stability result for A_n}
If $A\subseteq A_n$ is product-free and $100$-global
then $\mu(A) < e^{-cn^{1/3}}$ for some absolute constant $c>0$.
\end{thm}

More generally, we consider the following product mixing property.

\begin{defn} \label{def:prodmix}
We say that a triple $(A,B,C)$ of subsets of a group $G$
is  \emph{$\eps$-mixing for products} 
if for independent uniformly random  $a,b\sim G$ we have
\[
\Pr[a\in A,b\in B,ab\in C] \in (1-\eps,1+\eps) \mu(A)\mu(B)\mu(C).
\]

We say that $G$ is \emph{$\eps$-mixing for $\mu$-dense products}
if $(A,B,C)$ is $\eps$-mixing for products 
for any $A,B,C \sub G$ of density at least $\mu$.

We say that $G$ is  \emph{$\eps$-mixing for $r$-global $\mu$-dense products}
if $(A,B,C)$ is $\eps$-mixing for products 
for any $r$-global $A,B,C \sub G$ of density at least $\mu$.
\end{defn}

Gowers \cite{gowers} proved that $A_n$ is $\eps$-mixing for $O_\eps(n^{-1/3})$-dense products,
which is sharp by an example of Eberhard \cite{eberhard}. The sharpness example has dictatorial structure,
so one might hope that globalness would give mixing for much sparser sets.
We show that this indeed holds. Moreover, the exponent of the stretched exponential
(curiously also $1/3$) is sharp (see Section \ref{sec:Sharpness}).

\begin{thm} \label{thm:The global mixing property_Intro} 
The alternating group $A_{n}$ is $0.01$-mixing 
for $100$-global $e^{-cn^{1/3}}$-dense products,
for some absolute constant $c>0$.
\end{thm}

We note that Theorem \ref{cor:stability result for A_n} follows immediately from
Theorem \ref{thm:The global mixing property_Intro}. Another consequence
(consider triples $(A,A,xA^{-1})$ for any $x \in A_n$) is the following result
that global sets of density $e^{-O(n^{1/3})}$ have covering number at most $3$.

\begin{thm}
For some absolute constant $c>0$,
if $A\subseteq A_n$ is $100$-global 
and $\mu(A)\ge e^{-cn^{1/3}}$ then $A^3=A_n$.
\end{thm}

Clearly such a result does not hold for $A^2$ even when $A$ is quite dense:
consider any global $A$ and delete $A \cap (xA^{-1} \cup A^{-1}x)$
for some suitably chosen $x$.

\subsection{Further developments}

As mentioned above, a future paper of Lifshitz and Marmor~\cite{lifshitz2023bounds} 
applies our hypercontractivity result to improve various results of Larsen and Shalev 
on character bounds and mixing for conjugacy classes. They also obtained bounds for Kronecker coefficients, and upper bounded the size of normal product free sets.
A future paper of  Evra, Kindler, and Lifshitz~\cite{Evra2022Hypercontractivity} obtains analogous results 
for  the special linear group $\mathrm{SL}_n(\mathbb{F}_q)$. 
One can naturally generalise the above notions of globalness and level 
to any transitive subgroup $G$ of $S_n$, and then it is interesting to investigate 
whether there is an effective level $d$ inequality for global functions.
Perhaps it may be possible to extend it to all $2$-transitive subgroups?

\subsection{Organisation of the paper}

We will divide this paper into two parts, 
where the first part develops the general theory,
which is then applied in the second part 
to the various applications discussed above.

\newpage

\part{Hypercontractivity}

This part develops the general theory of sharp hypercontractivity in $S_n$. 
We start in the next section by gathering
various preliminaries from Representation Theory, Analysis and Algebraic Combinatorics.
The second section develops our main new ideas, 
which we use to prove our hypercontractivity result
Theorem \ref{thm:hypercontractivity for global functions in the symmetric group}.
In the last section of the part we use similar methods to prove
our level $d$ inequality Theorem \ref{thm:level-d for biglobal functions}
and deduce Theorem \ref{thm:level-d for global functions_intro}.

\section{Preliminaries}

Here we collect various background results from
Representation Theory (that of the symmetric group),
Algebraic Combinatorics (the spectrum of the Kneser graph),
and Analysis (on product spaces).

\subsection{Representations of the symmetric group}

Here we summarise the basic representation theory of the symmetric group,
referring to the book \cite{james2006representation} for the general theory
and following Ellis, Friedgut and Pilpel \cite{ellis2011intersecting} for its combinatorial interpretation.

Any conjugacy class of $S_n$ corresponds to a partition $\lL \vdash n$,
listing in non-increasing order the cycle lengths of any permutation in the class.
These partitions also index the Specht modules,
which are the ordinary (non-modular) irreducible representations of $S_n$.
These are constructed as follows. For any partition $\lL = (\lL_1,\dots,\lL_k)$ of $n$,
the Young diagram is an array of left-justified boxes, with $\lL_i$ boxes in the $i$th row.
A $\lL$-tableau is obtained by putting the numbers $\{1,\dots,n\}$ in these boxes.
A $\lL$-tabloid is a row equivalence class of  $\lL$-tableaux, where we say two tableau
are row equivalent if they have the same (unordered) set of numbers in each row.
We consider the natural left action of $S_n$ on  $\lL$-tabloids
and write $M^\lL$ for the resulting permutation module.
The Specht module $S^\lL$ is the submodule of $M^\lL$
generated by all elements $\kK_t \{t\}$ where $t$ is a $\lL$-tableau
and $\kK_t = \sum \text{sign}(\sS) \sS \in \mb{Z}[S_n]$,
where the sum is over all $\sS$ fixing the columns of $t$.

We identify $L^2(S_n)$ with the $S_n$-module $\mb{R}[S_n]$ (the regular representation).
The symmetric group $S_{n}$ acts on itself by left and right multiplication.
We embed $S_n$ in $[n]^n$ by identifying $\tau \in S_n$ with $(\tau(i): i \in [n]) \in [n]^n$.
The left action on $S_n$ or $[n]^n$ is $\sS \mapsto L_\sS$,
where $L_\sS(x)=(\sS(x_1),\dots,\sS(x_n))$ permutes the values of the coordinates;
the right action is $\sS \mapsto R_{\sS^{-1}}$,
where $R_\sS(x)=(x_{\sS(1)},\dots,_{\sS(n)})$ permutes the order of the coordinates.
The corresponding dual actions on $L^2(S_n)$
are $^\sS f(x) = f(\sS^{-1} x)$ and $f^\sS(x) = f(x\sS)$.

Viewing $L^2(S_n)$ as an $S_n \times S_n$-module
(via the above two commuting $S_n$ actions),
it decomposes into irreducible $S_n\times S_n$-modules $V_{=\lL} = V_\lL \otimes V_\lL^*$;
these are the isotypic components of $L^2(S_n)$ as an $S_n$-module,
each decomposing (non-canonically) into $\dim V_\lL$ copies
of the irreducible $S_n$-module $V_\lL$.

By Schur's Lemma, any operator $T$ on $L^2(S_n)$
that commutes with the action of $S_n$ from one side 
(such as the convolution operators $T_f$ above) preserves each isotypic component;
furthermore, if $T$ commutes with the action of $S_n$ from both sides
then $T$ acts as a scalar matrix on each isotypic component.

Recall (see Definition \ref{def:degdec}) that we have an orthogonal degree decomposition
$L^2(S_n) = \bigoplus_d V_{=d}$ and for any $f \in L^2(S_n)$ we write $f = \sum_d f^{=d}$,
where each $f^{=d}$ is the orthogonal projection of $f$ on $V_{=d}$.
This is refined by the orthogonal decomposition $L^2(S_n) = \bigoplus_\lL V_{=\lL}$,
where $V_{=\lambda}\le V_{=d}$ if and only if $\lL_1 = n-d$;
we call $n-\lambda_1$ the \emph{strict level} of $\lambda$.

For any $f \in L^2(S_n)$ we let $\tilde{f} = \text{sign} \cdot f$
be obtained by multiplication by the sign character.
This satisfies $(\tilde{f})^{=\lL} = \text{sign} \cdot f^{=\lL'}$,
where $\lL'$ is the conjugate partition whose rows
have the same sizes as the columns of $\lL$.
In particular, if $f$ is a function on $A_n$, also regarded 
as a function on $S_n$ that is zero on odd permutations,
then we have $\tilde{f}=f$, so $f^{=\lL} = \text{sign} \cdot f^{=\lL'}$.

The \emph{level} of $\lL$ is $d(\lL) := \min\{n-\lL_1,n-\lL'_1\}$.
We write $W_{=d} = \bigoplus_{d(\lL)=d} V_{=\lL}$.

\subsection{Juntas}

A function $f$ defined on a product space $\OO = \prod_{i=1}^n \OO_i$
is called an \emph{$I$-junta} for some $I \sub [n]$ if $f(x)$ only depends on $x_I = (x_i: i \in I)$.
Analogously, a function $f$ on $S_n$ is called an \emph{$I$-junta} 
if $f(\sS)$ only depends on $(\sS(i): i \in I)$;
equivalently, $f$ is invariant under the right action of the subgroup $U_I$,
meaning that $f(\sS \tT) = f(\sS)$ for all $\tT$ fixing $I$ pointwise.
We note that $V_{\le d}$ is the span of the $d$-juntas, i.e.\ the $I$-juntas with $|I|=d$.

For any right $S_n$-module $M$, we call $x \in M$ an $I$-junta if $xU_I = x$.
We write $V_{\le d}(M)$ for the span of  the $d$-juntas, i.e.\ the $I$-juntas with $|I|=d$.
This is an $S_n$-submodule, generated by the $I$-juntas for any fixed $I$,
as if $xU_I = x$ then $x\sS U_{\sS^{-1}(I)} = x\sS$.
We let $V_{=d}$ be the quotient module $V_{\le d}(M)/V_{\le d-1}(M)$.
If $M$ is equipped with an $S_n$-invariant inner product then we identify $V_{=d}(M)$ 
with the orthogonal complement of $V_{\le d-1}(M)$ in $V_{\le d}(M)$.

We caution the reader that these two notions 
of junta do \emph{not} agree for the $S_n$-module $[n]^n$;
here we adopt the first notion, which is the stronger of the two. 

We described above the orthogonal decompositions
$L^2(S_n) = \bigoplus_d V_{=d}$ and its refinement
$L^2(S_n) = \bigoplus_\lL V_{=\lL}$.
The following lemma generalises this to any $S_n$-module $M$,
with $V_{=\lL}(M)$ equal to the $\lL$-isotypic component of $M$.
The proof will use the following case of Young's branching rule:
for any $\lL \vdash n$, the restriction of $V_\lL$ to $S_{n-1}$
decomposes as $\bigoplus_\mu V_\mu$ where the sum is over
all $\mu \vdash n-1$ with Young diagram obtained 
by deleting a box from that of $\lL$.

\begin{lem}
\label{lem:rep-theory lemma} 
For any $S_n$-module $M$ we have $V_{=d}(M) = \bigoplus_{\lL: \lL_1=n-d} V_{=\lL}(M)$.
\end{lem}
\begin{proof}
It suffices to show $V_{\le d}(M) = \bigoplus_{\lL: \lL_1 \ge n-d} V_{=\lL}(M)$.

Recall that $V_{\le d}(M)$ is the $S_n$-module generated by 
$M^{U_I} := \{x: xU_I=x\}$ for any fixed $d$-set $I \sub [n]$.
Writing $M[U_I]$ for the restriction $M$ to a $U_I$-module,  we have $M^{U_I} = V_0(M[U_I])$: 
the isotypic component of the trivial representation of $U_I$,
which has Young diagram equal to a single row of length $n-d$.
By Young's branching rule, for any $\lL$ with $\lL_1 < n-d$
we deduce $(V_{=\lL}(M))^{U_I} = V_0(V_{=\lL}(M)[U_I]) = 0$.
Thus  $V_{\le d}(M) \le \bigoplus_{\lL: \lL_1 \ge n-d} V_{=\lL}(M)$.

To see that this inclusion is an equality, it suffices to show for each irreducible $X \le M$
where $X \cong V_\lL$ with $\lambda_1 \ge n-d$ that $V_{\le d}(X) = X$.
Now $V_{\le d}(X)$ is generated by $X^{U_I} = V_0(X[U_I])$,
which is nonzero by Young's branching rule.
As $V_{\le d}(X)$ is a subrepresentation of the irreducible $X$
we deduce $V_{\le d}(X)=X$. Summing over $X$, we deduce the lemma.
\end{proof}

\subsection{The spectrum of the Kneser graph} \label{subsec:Kneser}

The Kneser graph $K_{n,k}$ has vertex set $\binom{[n]}{k} := \{A \sub [n]: |A|=k\}$,
where $AB$ is an edge for $A,B \in \binom{[n]}{k}$ if $A \cap B = \es$.
We assume $k \le n/2$ for non-triviality. 
Every vertex in $K_{n,k}$ has degree $\tbinom{n-k}{k}$.
The random walk on the Kneser graph corresponds to 
the operator $\mc{K}=\mc{K}_{n,k}$ on $L^2(\binom{[n]}{k})$ 
given by $\mc{K}f(A)=\mb{E}_{B\sim \binom{[n]\sm A}{k}}[f(B)]$.

The eigenvalues of the Kneser graph were calculated by Lov\'asz  \cite{lovasz1979shannon}.
We will also need the correspondence between these eigenvalues and their representation level, 
as described by Bannai and Ito \cite[Section III.2]{bannai1984algebraic}.
Here, for any $S_n$-morphism $T$ of an $S_n$-module $M$,
we say that an eigenvalue of $T$ has \emph{level} $d$
if its corresponding eigenspace intersects $V_{=d}(M)$ non-trivially.

We regard $L^2(\binom{[n]}{k})$ as an $S_n$-module via $^\sS f(A)=f(\sS(A))$
and note that $\mc{K}$ is an $S_n$-morphism 
(as disjointness of subsets does not depend on the labelling of the underlying set),
so the eigenspaces of $\mc{K}$ are $S_n$-subrepresentations of $L^2(\binom{[n]}{k})$. 
We can naturally identify $L^2(\binom{[n]}{k})$ with the permutation module $M^{(n-k,k)}$,
in which the highest degree $V_{=k}(M^{(n-k,k)})$ is the Specht module $S^{(n-k,k)}$,
generated by elements  $\kK_t \{t\}$ with $\kK_t \in \mb{Z}[S_n]$  
of the form $\prod_{i=1}^k ( \text{id} - (a_i b_i) )$,
which corresponds under the identification with $L^2(\binom{[n]}{k})$
to a signed $k$-octahedron $O \in  \mb{Z}^{\binom{[n]}{k}}$ 
(a $k$-partite $k$-graph with parts $\{a_i,b_i\}_{i=1}^k$).
Applying $\mc{K}$ to $O$ we see that any $k$-set $A$ with $O_A=0$
receives equal $\pm$ contributions that cancel to zero,
whereas any $k$-set $A$ with $O_A \in \pm 1$ 
receives a single contribution $(-1)^k \tbinom{n-k}{k}^{-1} O_A$
from the unique $B$ disjoint from $A$ with $O_B \ne 0$.
Thus we obtain the following lemma.

\begin{lem}\label{lem:eigenvalues of the Kneser graph}
The unique level $k$ eigenvalue of $\mc{K}_{n,k}$ is $(-1)^k \tbinom{n-k}{k}^{-1}$.
\end{lem}

Now we will lift the above analysis to the setting of ordered $k$-tuples $[n]_k$,
on which we define the disjointness graph $D_{n,k}$, 
joining two $k$-tuples $I,J \in [n]_k$ if their underlying sets are disjoint.
Similarly to above, we write $\mc{D} = \mc{D}_{n,k}$ 
for the corresponding operator on $L^2([n]_k)$,
naturally regarded under the left action of $S_n$
as the permutation $S_n$-module $M^{(n-k,1^k)}$,
on which $\mc{D}$ is an $S_n$-morphism.

\begin{lem}\label{lem:Kneser-type operator}
If $n\ge 2k$ then every level $k$ eigenvalue $\aA$ of $\mc{D}_{n,k}$ 
satisfies $|\aA| \le (2k/n)^k$.
\end{lem}

\begin{proof}
Consider the $S_n$-morphism $\varphi: L^2([n]_k) \to L^2(\binom{[n]}{k})$ where $\varphi(f)(A)$ is the average of $f(x)$ over all orderings $x$ of $A$. 
The adjoint $S_n$-morphism $\varphi^*$ embeds $L^2(\binom{n}{k})$ as the subspace $L^2([n]_k)^{S_k}$ of $S_k$-invariant functions on $[n]_k$,
considering the right action of $S_k$ on $[n]_k$ by permuting coordinates. As $\varphi, \varphi^*$ are $S_n$-morphisms
they restrict to operators between $V_{\le k}( L^2([n]_k) )$ and $V_{\le k}( L^2(\binom{[n]}{k}) )$.
We note that $\mathcal{D}_{n,k} = \varphi^* \mathcal{K}_{n,k}\varphi.$
By Lemmas \ref{lem:eigenvalues of the Kneser graph} and  \ref{lem:rep-theory lemma} 
we deduce that $\aA$ is $0$ or $(-1)^k \tbinom{n-k}{k}^{-1}$.
As $\tbinom{n-k}{k} \ge (\tfrac{n-k}{k})^k \ge (n/2k)^k$ the lemma follows.
\end{proof}

\subsection{Analysis on product spaces I} \label{subsec:AnI}

In this subsection and the next we discuss various aspects of analysis on product spaces.
We consider the space $L^2\left([n]^n\right)$ of real valued functions on $[n]^n$
with uniform measure and the resulting inner product and $L_p$-norms.

We start with the Efron--Stein decomposition and noise operator, following \cite[Chapter 8]{o2014analysis}.
We call a function on $[n]^n$ a $d$-\emph{junta} if it only depends on $d$ variables. 
The space $L^2([n]^n)$ is equipped with a degree filtration. 
The \emph{degree} of $f\in L^2([n]^n)$ is the minimal $d$ 
for which $f$ is a linear combination of $d$-juntas. 
We let $V_{\le d}$ be the span of all functions of degree at most $d$.
We write $f^{\le d}$ for the orthogonal projection of $f$ on $V_{\le d}$ 
and $f^{=d} = f^{\le d} - f^{\le d-1}$. 
This gives rise to an orthogonal decomposition 
$f= \sum_{d=0}^{n}f^{=d}$.

Next we describe a finer decomposition known as the Efron--Stein decomposition. 
For $S\subseteq [n]$ we let $V_{\le S}$ be the space of functions that depend only on the coordinates of $S$.
We let $V_{=S}$ be the space of functions in $V_{\le S}$ that are othogonal to all functions of degree $\le d-1.$ 
Then the spaces $V_{=S}$ form an orthogonal decomposition of $L^2([n]^n).$ 
We write $f^{=S}$ for the projection of $f$ onto $V^{=S}.$
The  \emph{Efron--Stein} decomposition is $f=\sum_{S\subseteq [n]} f^{=S}$.

For $\rho \in (0,1)$ and $x \in [n]^n$ we define a distribution  $y\sim N_{\rho}(x)$
where for each $i \in [n]$ independently with probability $\rho$ we let $y_i=x_i$
or otherwise we sample  $y_i \sim [n]$ uniformly at random.
The associated noise operator $T_\rho$ on $L^2([n]^n)$ 
is defined by $T_\rho f(x) = \mb{E}_{y\sim N_{\rho}(x)} f(y)$.
In terms of the Efron--Stein decomposition
we have $T_{\rho} f  = \sum_{d=0}^{n}\rho^{d}f^{=d}$.

It will be helpful to recall the proof of this last fact,
as it will motivate our approach to estimating the eigenvalues 
of the noise operators constructed in the proof of 
Theorem~\ref{thm:hypercontractivity for global functions in the symmetric group}.
Consider any $f\in V_{=S}$ and $x\in [n]^n$.
Let  $z\in\{0,1\}^n$ with $\mb{P}(z_i=1)=\rho$ independently.
Let $y$ be obtained from $x$ by resampling those $x_i$ with $z_i=1$.
Then $T_\rho f(x) = \mb{E} f(y)$. 
Let $S' = \{i \in S: z_i=1\}$.
If $S'=S$ then $y=x$.
Otherwise, letting $g$ be the $S'$-junta $g(y)=n^{|S'|}1_{y_{S'} = x_{S'}}$
we have $\mb{E}[f(y)|z] = \langle f,g\rangle = 0$.
Thus $T_\rho f(x) = \mb{E}_z \mb{E}[f(y)|z] = \mb{P}(S'=S) f(x) = \rho^{|S|} f(x)$. 
We deduce $T_{\rho} f  = \sum_{d=0}^{n}\rho^{d}f^{=d}$.

\subsection{Analysis on product spaces II}  \label{subsec:AnII}

We continue our discussion of analysis on $L^2([n]^n)$ 
with global hypercontractivity and the level $d$ inequality.
Here the context is that many applications of the classical hypercontractive inequality
break down in general product spaces (even for Boolean functions with bias)
and that there is a substantial literature of deep results 
addressing the ineffectiveness of this inequality on juntas
(see \cite{friedgut1999sharp, hatami2012structure}).
A breakthrough by Keevash, Lifshitz, Long and Minzer \cite{KLLM}
gave a strong resolution of this problem by proving an effective 
hypercontractive inequality for global functions on general product spaces,
thus obtaining a tight form of Bourgain's sharp threshold theorem
and a variant form of the Kahn-Kalai Isoperimetric Conjecture
and a general form of the Invariance Principle of Mossel, O'Donnell and Oleszkiewicz.
There are many subsequent applications (surveyed in \cite[Section 1.6]{KLLM})
to hypercontractivity in Exotic Settings, Extremal Combinatorics, 
Product-free sets (discussed above), Error-correcting Codes,
High-dimensional Expanders and variants of the Unique Games Conjecture.

Despite these many applications, the result of  \cite{KLLM} is not sharp in terms of the noise rate, 
which precludes its use for applications as in the current paper where this is a crucial point.
Instead, we will rely on a more recent improvement by Keller, Lifshitz and Marcus \cite{keller2023}.

We require the following key definition, which will be used throughout the remainder of the paper.
It subsumes Definition \ref{def:global} on global functions in $L^2(S_n)$,
giving a common description of (bi)globalness on $L^2(S_n)$ and $L^2([n]^n)$.
We recall that globalness for $f \in L^2(S_n)$ was defined 
using functions $f_{I\to J}$, which we call a $d$-restriction, where $d=|I|=|J|$.
Similarly, for $f \in L^2([n]^n)$ we will use $d$-restrictions of the form $f_{S \to x}$,
where $x \in [n]^S$ for $S \sub [n]$ with $|S|=d$, defined by
$f_{S\to x}(y) = f(x,y)$ (aligning $x$ with the $S$ coordinates).

\begin{defn} \label{def:biglobal}
Let $f \in L^2(S_n)$ or $f \in L^2([n]^n)$.
We say that $f$ is $(r,\gG_2,d)$-\emph{global} 
if $\|f_*\|_2 \le r^t \gG_2$ for all $t$-restrictions $f_*$ of $f$ with $t \le d$.

We say that $f$ is $(r,\gG_1,\gG_2,d)$-\emph{biglobal} 
if $\|f_*\|_1 \le r^t \gG_1$ and $\|f_*\|_2 \le r^t \gG_2$ 
for all $t$-restrictions $f_*$ of $f$ with $t \le d$.

If we omit any parameter from the notation it is understood as the default option:
the default for $d$ is $n$, the default for $\gG_i$ is $\|f\|_i$.
We use the same terminology for a set $A$,
meaning that $1_A$ has the corresponding property.
\end{defn}

Now we can state the hypercontractivity theorem and level $d$ inequality 
proved by Keller, Lifshitz and Marcus \cite{keller2023}.

\begin{thm}[{\cite{keller2023}}]\label{thm:KLM}
Suppose $f\in L^2([n]^n)$ is $(r,\gG)$-global with $r>1$ and $\gG>0$.
If $q\ge 2$ and $\rho = \frac{\log q}{16rq}$
then $\|T_{\rho}f\|_q^q \le \|f\|_2^2\gamma^{q-2}$.
 \end{thm}

\begin{thm}[\cite{keller2023}]\label{thm:KLM2} 
Suppose $f\in L^2([n]^n)$ is $(r,\gG_1,\gG_2,d)$-biglobal
with $r>1$, $\gG_2>\gG_1>0$ and $2d \le \log(\gG_2/\gG_1)$.
Then $\| f^{=d} \|_2^2 \leq \gG_1^2 ( 2200 r^2 d^{-1} \log(\gG_2 /\gG_1) )^d$.
 \end{thm}

We note that if $f=1_A$ is $(r,d)$-biglobal then Theorem \ref{thm:KLM2}
applies with $\gG_i = \|f\|_i = \mu(A)^{1/i}$ for $i=1,2$,
so the factor $\log(\gG_2/\gG_1) = \tfrac{1}{2}\log(1/\mu(A))$
generalises the $\log(1/\mu(A))$ factor appearing
in the results stated in the introduction.

\section{Hypercontractivity}

In this section we prove our hypercontractivity result,
Theorem \ref{thm:hypercontractivity for global functions in the symmetric group}.
The idea is to couple the symmetric group $S_n$ with the product space  $[n]^{n}$
and exploit the sharp hypercontractivity theorem for global functions on $[n]^{n}$
(Theorem \ref{thm:KLM}) proved by Keller, Lifshitz and Marcus \cite{keller2023}.
Our coupling will allow us to define a noise operator on $S_n$ that inherits 
similar properties from the standard noise operator on $[n]^n$.
The definition of the coupling is fairly straightforward,
but the definition and analysis of the noise operator is quite subtle:
it is here that the main technical difficulties of this paper will arise.

\subsection{Coupling}

In this subsection we will define a coupling between $S_n$ and $[n]^n$,
which will commute with the left action of $S_n$,
meaning that it depends on the order of the coordinates
but not on the values taken by the coordinates. 

We start with some general preliminaries on couplings.
Given two finite probability spaces $(X,\mu_X)$ and $(Y,\mu_Y)$,
a \emph{coupling} $\mu$ (of $X$ and $Y$) is a probability distribution on $X \times Y$
with marginal distributions $\mu_X$ on $X$ and $\mu_Y$ on $Y$.
We also think of $\mu$ as a weighted bipartite graph 
with parts $X$ and $Y$ and edge weights $\mu(x,y)$.
Hence we obtain a random walk where the probability
of moving from $x \in X$ to $y \in Y$ is $\mu(x,y)/\mu_X(x)$;
we denote this distribution by $y\sim N(x)$.

We define the operator $T_{X\to Y}: L^2(X)\to L^2(Y)$
associated with the coupling $\mu$ by 
$T_{X\to Y}f(y) = \mathbb{E}_{x\sim N(y)}[f(x)]$.
Reversing the roles of $X$ and $Y$ we similarly obtain an operator
$T_{Y\to X}: L^2(Y)\to L^2(X)$, which is adjoint to $T_{X\to Y}$.
Jensen's inequality implies that these operators are contractions in $L^p$
for any $p \ge 1$.

It will be helpful to describe our coupling in two equivalent ways.
For the first description, which we call the \emph{greedy coupling},
we choose $x\sim [n]^{n}$
and turn it into a permutation $\sS$ by looking at the coordinates
of $x$ one by one and setting $\sS(i)=x_{i}$ if possible.
The formal definition is as follows.

\begin{defn}[Greedy coupling]
Choose $x\sim [n]^{n}$ uniformly at random. 
For $i=1,\dots,n$ sequentially we define $\sS(i)$ as follows.
If $x_i \notin \{\sS(1),\dots,\sS(i-1)\}$ we let $\sS(i)=x_i$.
Otherwise, we choose $\sS(i)$ in $[n] \sm \{\sS(1),\dots,\sS(i-1)\}$
 uniformly at random. We write $\sS \sim N(x)$ for $\sS$ given $x$
and define $T_C:L^2(S_n) \to L^2([n]^n)$
by $T_C f(x) = \mb{E}_{\sS \sim N(x)} f(\sS)$.
\end{defn}

Our second description. which we call the \emph{forgetful model},
operates in the reverse direction, starting with $\sS \sim S_n$ 
and letting each $\sS(i)$ be either remembered, when we set $x_i=\sS(i)$, or forgotten,
when we sample $x_i$ uniformly from the conflicting set $\{\sS(1),\dots,\sS(i-1)\}$.
We choose the probability of remembering to make the
marginal distribution of $x$ uniformly random in $[n]^n$.
  
\begin{defn}[The forgetful model]
Let $\sigma\sim S_{n}$ and $z_1,\dots,z_n \sim \{0,1\}$ be independent
with $\sS$ uniform and each $\mb{P}(z_i=1)=\tfrac{n-i+1}{n}$.
If $z_i=1$ we let $x_i=\sS(i)$, or if $z_i=0$ we choose
$x_i$ uniformly in $\{\sS(1),\dots,\sS(i-1)\}$.
We write $x \sim N(\sS)$ for $x$ given $\sS$
and define $T^*_C: L^2([n]^n) \to L^2(S_n) $
by $T^*_C f(x) = \mb{E}_{x \sim N(\sS)} f(x)$.
\end{defn}

The following lemma justifies our assertion that these descriptions are equivalent.

\begin{lem}
\label{lem:noisy model} 
Let $\sigma\sim S_n,x\sim N(\sigma)$ be chosen according to the greedy coupling 
and let $y \sim [n]^n,\tau\sim N(y)$ be chosen according to the forgetful model. 
Then $\left(\sigma,x\right)$ and $(\tau,y)$ have the same distribution.
\end{lem}

\begin{proof}
It suffices to show for each $i$ that both models have the same 
conditional distibution of $(x_i,\sS(i))$ given all $(x_j,\sS(j))$ with $j<i$.
In both models the distribution of  $(x_i,\sS(i))$
only depends on the conflicting set $C_i := \{\sS(1),\dots,\sS(i-1)\}$
and by symmetry the marginal distribution of $\sS(i)$ is uniform in $[n] \sm C_i$.
Thus it suffices to show that both models have the same
conditional distibution $\mb{P}'$ of $x_i$ given $\{\sS(1),\dots,\sS(i)\}$.

In the forgetful model, by definition we have $\mb{P}'(x_i=\sS(i)) = \tfrac{n-i+1}{n}$
and $\mb{P}'(x_i=\sS(j)) = \tfrac{1}{n}$ for each $j<i$.
In the greedy model we always have $x_i \in \{\sS(1),\dots,\sS(i)\}$,
with $\mb{P}'(x_i=\sS(i)) = \mb{P}'(x_i \notin C_i) =  \tfrac{n-i+1}{n}$,
and otherwise $x_i$ is uniform on $C_i$ by symmetry. The lemma follows.
\end{proof}

We write $C(\sS,x)$ for the above coupling on $S_n \times [n]^n$.
Now we justify our assertion that $C$ depends only on the order of the coordinates,
not their values, so that the associated operators commute with the left action of $S_n$.
We note that this also holds for the noise operator $T_\rho$ on $[n]^n$.
 
\begin{lem}
\label{lem:Coupling operators are left morphisms} 
The coupling $C$ is left $S_n$-invariant and
the operators $T_C$, $T^*_C$ and $T_\rho$ are left $S_n$-morphisms.
\end{lem}

\begin{proof}
The left invariance $C(\sS,x) = C(\tau\sS,\tau x)$ for all $\tau\in S_{n}$
holds as the definition of the coupling is invariant under renaming coordinate values.

To deduce that $T_C$ is a left $S_n$-morphism we calculate
\[
\tau\left(T_{C}f\right)\left(x\right)=T_{C}f\left(\tau^{-1}x\right)
=\mathbb{E}_{\sigma\sim N\left(\tau^{-1}x\right)}\left[f\left(\sigma\right)\right]
=\mathbb{E}_{\sigma\sim N\left(x\right)}\left[f\left(\tau^{-1}\sigma\right)\right]=T_{C}\left(\tau f\right).
\]
We omit the similar calculations for $T_C^*$ and $T_\rho$.
\end{proof}

\subsection{Our noise operator on $S_n$}

Lemma \ref{lem:Coupling operators are left morphisms} implies that
$\tilde{T}_{\rho}:=T_{C}^{*} T_{\rho} T_{C}$ 
is a left $S_{n}$-morphism on $L^2(S_n)$. 
A symmetrization trick turns it into an operator\footnote{
We overload the notation for $T_\rho$ on $[n]^n$,
as it will be clear from the context which is intended.
} 
\[ T_\rho := \mb{E}_{\sS \sim S_n} [ R_\sS^* \tilde{T}_{\rho} R_\sS ] \] 
that commutes with the action of $S_n$ from both sides, as required for 
part (3) of Theorem \ref{thm:hypercontractivity for global functions in the symmetric group}.
Here we recall that $R_{\sigma}$ is the operator $f(x) \mapsto f(x\sS)$,
noting that its adjoint is $R_\sS^*=R_{\sS^{-1}}=R_\sS^{-1}$
and it satisfies $R_\sS R_\tau = R_{\sS \tau}$.

\begin{lem}
\label{lem:Part 3 of hypercontractivity} The operator $\mathrm{T}_{\rho}$ on $L^2(S_n)$
commutes with the action of $S_{n}$ from both sides.
\end{lem}

\begin{proof}
To see that $\mathrm{T}_{\rho}$ is a left $S_{n}$-morphism,
we note that this holds for $\tilde{T}_{\rho}$, $R_{\sigma}$ and $R_{\sigma}^{*}=R_{\sigma^{-1}}$.
To complete the proof we show that it is a right morphism:
\begin{align*}
R_\tau T_\rho & = R_\tau \mb{E}_\sS[  R_\sS^* T_\rho R_\sS ]
= \mb{E}_\sS [ R_{\tau\sS^{-1}} T_\rho R_{\sS\tau^{-1}} R_\tau ]
= T_\rho R_\tau. \qedhere
\end{align*}
\end{proof}

\subsubsection*{Globalness is preserved} 
In preparation for proving part (2) of Theorem \ref{thm:hypercontractivity for global functions in the symmetric group},
we show that globalness is preserved by the coupling.

\begin{lem}
\label{lem:globalness is preserved by T_C} 
Suppose that $f$ is $\left(r,\gamma\right)$-global.
Then $T_C f$ is also $\left(r,\gamma\right)$-global. 
\end{lem}

\begin{proof}
Let $S\subseteq\left[n\right]$ and $a\in\left[n\right]^{S}$.
With $b \sim [n]^{S^c}$, by Cauchy--Schwarz we have 
\begin{align*}
\|\left(T_{C}f\right)_{S\to a}\|_{2}^{2} & =\mathbb{E}_b \left[T_{C}f\left(a,b\right)\right]^{2}\\
 & =\mathbb{E}_b \left|\mathbb{E}_{\sigma\sim N\left(a,b\right)}\left[f\left(\sigma\right)\right]\right|^{2}\\
 & \le\mathbb{E}_{b}\mathbb{E}_{\sigma\sim N\left(a,b\right)}\left|f\left(\sigma\right)\right|^{2}.
\end{align*}
To prove the lemma, it suffices to show that the above distribution of $\sS \sim N(a,b)$
is a mixture of uniform distributions in $t$-umvirates with $t \le |S|$.
We will show conditionally on $\{\sS^{-1}(a_i): i \in S\}$ that 
$\sS$ is uniform in $\bigcap_{i \in S} U_{\sS^{-1}(a_i) \to a_i}$.

To see this, let $U = \bigcap_{i \in S} U_{a_i \to a_i}$ and note for any $\tau \in U$ 
that $L_\tau (a,b)$ has the same distribution as $(a,b)$,
and so $\sS' \sim N(L_\tau (a,b))$ has the same distribution as $\sS \sim N(a,b)$.
By left $S_n$-invariance of the coupling $C$ we have $\sS' \stackrel{D}{=} L_\tau \sS$,
so $\sS$ is uniform in $\bigcap_{i \in S} U_{\sS^{-1}(a_i) \to a_i}$, as required.
\end{proof}%

The same proof gives a corresponding statement for biglobal functions.

\begin{lem}\label{lem:biglobalness is preserved under the coupling}
If $f \in L^2(S_n)$ is $(r,\gamma_1,\gamma_2,d)$-biglobal
then $T_C f$ is $(r,\gamma_1,\gamma_2,d)$-biglobal.
\end{lem}

\subsubsection*{Hypercontractivity}
We are now ready to prove part (2) of Theorem \ref{thm:hypercontractivity for global functions in the symmetric group}.

\begin{lem}
\label{lem:Part 2 of the hypercontractivity theorem }
If $f\in L^{2}\left(S_{n}\right)$ is $\left(r,\gamma\right)$-global, $q\ge 2$ and $\rho\le\frac{\log q}{16rq}$
then $\|\mathrm{T}_{\rho}f\|_{q}^{q}\le\gamma^{q-2}\|f\|_{2}^{2}$.
\end{lem}

\begin{proof}
By the triangle inequality it suffices to show
$\|R_{\sigma}^{*}\tilde{\mathrm{T}}_{\rho}R_{\sigma}f\|_{q}^{q}\le\gamma^{q-2}\|f\|_{2}^{2}$
for each $\sS \in S_n$. We note that $g := R_{\sigma}f$
is $\left(r,\gamma\right)$-global and $\|g\|_{2}^{2}=\|f\|_{2}^{2}.$ 
By invariance of the uniform measure on $S_{n}$ 
it suffices to show $\| \tilde{T}_\rho g \|_q^q \le\gamma^{q-2}\|g\|_{2}^{2}$.
 
Recall that $\tilde{T}_{\rho}:=T_{C}^{*} T_{\rho} T_{C}$,
where $T_C$ and $T_C^*$ are contractions in all $L^p$ norms with $p \ge 1$.
Furthermore, $T_C g$ is  $\left(r,\gamma\right)$-global by Lemma \ref{lem:globalness is preserved by T_C},
so the lemma  follows from Theorem \ref{thm:KLM}.
\end{proof}

\subsection{The eigenvalue bound}

The remainder of this section will address the main difficulty in the proof of 
Theorem \ref{thm:hypercontractivity for global functions in the symmetric group},
namely part (1), which is the eigenvalue bound.
In this subsection we prove this bound, assuming the following technical lemma
that will then become our target for the remainder of the section.
To indicate its utility, we remark that it suffices 
to verify the eigenvalue bound for a $T$-junta $f$,
where as we average over $\sS$ in the definition of our noise operator
it will only cost a factor exponential in $|T|$ to assume $T \sub [n/4,n/2]$;
the lemma shows for such $f \in V_{=d}$  (the pure degree $d$ functions on $S_n$)
that $T_C f$ has significant projection on $V_{\le d}(L^2([n]^n)$, so we can obtain
the required eigenvalue bound from that of the usual noise operator on $[n]^n$.

For the statement, we recall that $f$ in $L^2(S_n)$ or $L^2([n]^n)$
is a $T$-junta if its value only depends on the coordinates in $T$.
Given a $T$-junta $f$ on $S_n$, we define the \emph{associated} $T$-junta
$\tilde{f}$ on $[n]^n$ by letting $\tilde{f}(x)=0$ if $(x_i: i \in T)$ are not all distinct,
or otherwise $\tilde{f}(x)=f(\sS)$ for any/all $\sS \in S_n$ agreeing with $x$ on $T$.

\begin{lem}
\label{lem:large inner product}  
Suppose $d\le \frac{n}{10^5}$ and $f \in V_{=d}(L^2(S_n))$ 
is a $T$-junta with $|T|=d$ and $T \sub [n/4,n/2]$.
Let $\tilde{f} \in L^2([n]^n)$ be the associated $T$-junta of $f$.
Then $\langle T_C f, \tilde{f} \rangle \ge 3^{-d} \|f\|_{2}^{2}$.
\end{lem}

Assuming this lemma, we now complete the proof of our hypercontractivity theorem.

\begin{proof}[Proof of Theorem \ref{thm:hypercontractivity for global functions in the symmetric group}]
We proved part (2) as Lemma \ref{lem:Part 2 of the hypercontractivity theorem },
and part (3) as Lemma \ref{lem:Part 3 of hypercontractivity}. 
Thus it remains to show part (1), which states that if $f \in V_{\le d}$ with $d \le 10^{-5} n$ 
then $\langle \mathrm{T}_{\rho}f, f \rangle \ge (\rho/72)^{d} \|f\|_{2}^{2}$.
This is trivial for $d=0$, so we assume $d \ge 1$. 
We have $\langle T_\rho f, f \rangle \ge \aA \|f\|_2^2$, 
where $\aA$ is the minimum level $d$ eigenvalue of $T_\rho$,
which is non-negative since $\tilde{T}_{\rho} = (T_{\sqrt{\rho}} T_C)^*(T_{\sqrt{\rho}} T_C)$ is positive semidefinite. 
It suffices to show $\aA \ge (\rho/72)^d$.

We claim that the $\aA$-eigenspace of $T_\rho$ at level $d$ contains a $d$-junta $f$.
To see this, note that as $T_\rho$ commutes with the $S_n$ action,
it preserves the space $V_{T}$ of $T$-juntas for any $T$,
which therefore decomposes as a sum of eigenspaces of $T_\rho$.
Since $V_{=d}$ is spanned by juntas, its $\aA$-eigenspace is the sum of the
$\aA$-eigenspaces at level $d$ of each junta space $V_{T}$, so the claim follows.
Say $f$ is a $T$-junta.

For any $\tau \in S_n$ with $\tau(T) \sub [n/4,n/2]$,
applying Lemma \ref{lem:large inner product} to $R_\tau f$ 
gives $\langle T_C R_\tau f, R_\tau \tilde{f} \rangle \ge 3^{-d} \|f\|_{2}^{2}$.
The degree of $R_\tau \tilde{f}$ is at most $d$,
so by orthogonality, Cauchy-Schwartz and $\| \tilde{f} \| \le \| f \|$ we have
\[ \langle T_C R_\tau f, R_\tau \tilde{f} \rangle
 =   \langle (T_C R_\tau f)^{\le d}, R_\tau \tilde{f} \rangle
 \le \| (T_C R_\tau f)^{\le d} \| \| f \|. \]
Combining with the lower bound and rearranging yields
\begin{equation}\label{eq:4.1}
\|\left(T_{C}R_{\tau}f\right)^{\le d}\|_{2} \ge 3^{-d} \|f\|_{2}.
\end{equation}
As $\mb{P}_\tau( \tau(T) \sub [n/4,n/2] ) \ge \prod_{i=1}^d \tfrac{n/4 - i}{n-i} \ge \tfrac12 4^{-d}$,
using \eqref{eq:4.1} and the eigenvalues of the standard noise operator on $[n]^n$ we deduce
\begin{align*}
\alpha \| f \|^2 & = \left\langle \mathrm{T}_{\rho}f,f\right\rangle  
=\mathbb{E}_{\tau}\|\mathrm{T}_{\sqrt{\rho}}T_{C}R_{\tau}f\|_{2}^{2}\\
 & \ge\rho^{d}\mathbb{E}_{\tau}\|\left(T_{C}R_{\tau}f\right)^{\le d}\|_{2}^{2}
  \ge \tfrac12 (\rho/36)^{d}\|f\|_{2}^{2}.
\end{align*}
Thus we have the required bound $\aA \ge (\rho/72)^d$, so the lemma follows.
\end{proof}

\subsection{Spread couplings}

Our task now is bounding $\langle T_C f, \tilde{f} \rangle$
where $f \in L^2(S_n)$ is an $I$-junta with $I \sub [n/4,n/2]$
and $\tilde{f} \in L^2([n]^n)$ is its associated $I$-junta.
We will think of this inner product as $\mb{P}[\mc{E}] \langle Tg, g \rangle$, where
$g \in L^2([n]_d)$ is naturally defined by $g(J)=f(\sS)$ for any/all $\sS \in U_{I \to J}$,
and $T=T(\nu)$ is the operator on $L^2([n]_d)$ 
associated to the coupling $\nu(a,b)$ on $[n]_d \times [n]_d$
obtained from our coupling $C(\sS,x)$ on $S_n \times [n]^n$
by letting $a = (\sS(i): i \in I)$ and $b = (x_i: i \in I)$, 
conditioned on the event $\mc{E}$ that $b$ has distinct coordinates. 
This is indeed equivalent, as
\begin{equation} \label{eq:fg}
 \langle T_C f, \tilde{f} \rangle = \mb{E}_{(\sS,x) \sim C} [ f(\sS)\tilde{f}(x) ] 
= \mb{P}[\mc{E}]  \mb{E}_{(a,b) \sim \nu} [ g(a)g(b) ] = \mb{P}[\mc{E}]  \langle T(\nu)g, g \rangle. 
\end{equation}

Recalling that $C$ is left $S_n$-invariant, we see that $\nu$ is  left $S_n$-invariant,
i.e.\  $\nu(L_\sS a, L_\sS b) = \nu(a,b)$ for all $\sS \in S_n$.
Thus $T$, $T^*$ are $S_n$-endomorphisms of $[n]_d$,
and both marginals of $\nu$ are uniform on $[n]_d$,
so henceforth we equip $[n]_d$ with the uniform measure.
Now $\mb{P}(\mc{E}) = |[n]_d|/n^d$ and
\[p_{\text{lazy}} := \nu(\{a=b\}) = \sum_a \nu(a,a) = |[n]_d|\nu(a,a)\] 
does not depend on $a$.
By the forgetful model of the coupling, as $I \sub [n/4,n/2]$,
conditional on $\sS$ we have
\begin{equation} \label{eq:lazy}
 \mb{P}(\mc{E})  p_{\text{lazy}} = \mb{P}_{x \sim N(\sS)} \Big[ \bigcap_{i \in I} \{ x_i = \sS(i) \} \Big]
= \prod_{i \in I}  \frac{n-i+1}{n} \ge 2^{-d}.  
\end{equation}

To obtain a lower bound on $\langle Tg, g \rangle$, we will show that $Tg$ is close to $p_{\text{lazy}} g$.
We will obtain this from spreadness of $\nu$, in the sense of the following definition.

\begin{defn}
 We say that an $S_n$-invariant coupling $(a,b) \sim \nu$ on $[n]_d$ with $p_{\text{lazy}} := \nu(\{a=b\})$
 is $p$-spread if $\nu(a,b) \le p_{\text{lazy}} p^{|\{i:a_i\ne b_i\}|}$.
\end{defn}

Considering the forgetful model as above, clearly $\nu$ is $10/n$-spread, say.

\begin{lem}\label{lem:noisy coupling eigenvalues}
Let $g \in V_{=d}(L^2([n]_d))$ with $n \ge 10^5 d$ 
and $\nu$ be a $10/n$-spread $S_n$-invariant coupling on $[n]_d$.
Then $\| T(\nu)g - p_{\mathrm{lazy}} g\| \le .1p_{\mathrm{lazy}}\|g\|_2$.
\end{lem}

Assuming this lemma, we now deduce the required bound on $\langle T_C f, \tilde{f} \rangle$.
 
\begin{proof}[Proof of Lemma \ref{lem:large inner product}]
By \eqref{eq:fg} we have $\langle T_C f, \tilde{f} \rangle = \mb{P}[\mc{E}]  \langle T(\nu)g, g \rangle$,
with $g,T(\nu),\mc{E}$ defined as above. We may assume $f \in V_{=d}(L^2(S_n)$, and so $g \in V_{=d}(L^2([n]_d))$.

By Cauchy-Schwartz and Lemma \ref{lem:noisy coupling eigenvalues} we have
\begin{align*}
& |\langle Tg - p_{\mathrm{lazy}}g, g\rangle| \le \| T(\nu)g - p_{\mathrm{lazy}} g\| \|g\| \le .1p_{\mathrm{lazy}}\|g\|_2^2, \text{ so} \\
& \langle Tg , g\rangle  = p_{\mathrm{lazy}}\|g\|_2^2 + \langle Tg - p_{\mathrm{lazy}}g, g\rangle \ge .9p_{\mathrm{lazy}} \|g\|_2^2.
\end{align*}
By \eqref{eq:lazy} and $\|f\|_2=\|g\|_2$ we deduce $\langle T_C f, \tilde{f} \rangle \ge 3^{-d} \|f\|_{2}^{2}$.
\end{proof}

\subsection{The staying decomposition}

Going deeper into the heart of the proof, 
for the remainder of the section we now fix $T=T(\nu)$ 
for some general $10/n$-spread coupling $\nu$ on $[n]_d$ with $n \ge 10^5 d$ 
and show that $Tg$ is close to $p_{\mathrm{lazy}}g$ for any $g \in V_{=d}(L^2([n]_d))$.
It will be helpful now to recall our earlier discussion (see Section \ref{subsec:AnI})
of the eigenvalues of the noise operator $T_\rho$ on $[n]^n$. The idea was to condition on
the auxiliary random variable $z \in \{0,1\}^n$ that identifies the remembered coordinates in $T_\rho$.
There the analysis was very simple, as this conditioning produces an operator that is either the identity or zero.
Here we follow an analogous procedure, using the analysis of the Kneser graph  (see Section \ref{subsec:Kneser}) 
to show that our conditioning produces an operator that is either close to the identity or close to zero.

Our decomposition requires the following definition.

\begin{defn}
For any $a,b$ in $[n]_d$ we let $\text{stay}(a,b) := \{ i \in [d]: b_i \in \{a_1,\dots,a_d\} \}$.
Given a left invariant coupling $\nu'$ on $[n]_k$, 
we say that $\nu'$ and $T(\nu')$ are \emph{$S$-staying}
if $\nu'$ is supported on the event $A_S := \{ \text{stay}(a,b)=S \}$.
\end{defn}

The following lemma describes the staying decomposition. We omit the routine proof,
just remarking that by $S_n$-invariance of $\nu$, for any $a \in [n]_d$ we have 
$\mb{P}_{b \sim N_\nu(a)} [ A_S ] = \aA_S$ independent of $a$.

\begin{lem}\label{lem:decomposition into staying}
Let $\nu$ be a left invariant coupling on $[n]_d$.
Write $\nu = \sum_{S \sub [d]} \aA_S \nu_S$ and $T = \sum_{S \sub [d]} \aA_S T_S$, 
where $\aA_S = \nu(A_S)$, $\nu_S(a,b) = \nu(a,b)/\aA_S$, $T=T(\nu)$, $T_S=T(\nu_S)$.
Then each $T_S$ is left invariant and $S$-staying.
\end{lem}

We now state two technical lemmas concerning the staying decomposition.

\begin{lem}\label{lem:Upper bound on alpha_S}
Let $\nu$ be a $10/n$-spread left-invariant coupling on $[n]_d$ with $n \ge 10^5 d$.
For $S \sub [d]$ let $A_S := \{ \text{stay}(a,b)=S \}$.
Then $\nu(A_S) \le 1.001 p_{\mathrm{lazy}} 10^{d-|S|}$.
\end{lem}

\begin{lem}\label{lem:s-staying kill spec_k}
Let $f \in V_{=d}(L^2([n]_d))$ and  $T$ be a left-invariant $S$-staying operator on $L^2([n]_d)$.
Then $\|Tf\|_2 \le \left(\frac{2(d-|S|)}{n-|S|}\right)^{d-|S|}\|f\|_2$.
\end{lem}

We now complete the proof assuming these two technical lemmas.

\begin{proof}[Proof of Lemma \ref{lem:noisy coupling eigenvalues}]
Let $g \in V_{=d}(L^2([n]_d))$ with $n \ge 10^5 d$ 
and $\nu$ be a $10/n$-spread $S_n$-invariant coupling on $[n]_d$.
By Lemma \ref{lem:decomposition into staying} we invoke the staying decomposition
$T = T(\nu) = \sum \aA_S T_S$ for $S$-staying operators $T_S$,
with each $\aA_S \le 1.001p_{\mathrm{lazy}}10^{d-|S|}$
by  Lemma \ref{lem:Upper bound on alpha_S}
and $\|T_S g\|_2 \le \left(\frac{2(d-|S|)}{n-|S|}\right)^{d-|S|}\|g\|_2$
by Lemma \ref{lem:s-staying kill spec_k}.

Considering \[1.001p_{\mathrm{lazy}} \ge \aA_{[d]} = \mb{P}(\text{stay}(a,b)=[d])
= p_{\mathrm{lazy}} + \mb{P}(A'),\] where $A'=A_{[d]} \cap \{a \ne b\}$,
we can write $T_{[d]} = (1-\bB) Id + \bB T'$ with $T' = T(\nu | A')$
and $\bB = \mb{P}(A' | A_{[d]}) < 0.001$. As $T'$ is contractive, by the triangle inequality
\[\|T_{[d]}g-g\|_2 \le \bB\|g\|+\bB\|T'g\| \le 0.002\|g\|_2,\]
so \[ \|\aA_{[d]}T_{[d]}g - p_{\mathrm{lazy}} g\|_2 \le  0.003 p_{\mathrm{lazy}} \|g\|_2. \] 
  
We estimate the remaining terms, using $\binom{d}{i}\le \left(\frac{ed}{d-i}\right)^{d-i}$, as
\begin{align*}
 \sum_{S\ne [d]}\alpha_S\|T_S[g]\|_2 
  & \le \sum_{i=0}^{d-1} \left(\frac{ed}{d-i}\right)^{d-i} \cdot
    1.001p_{\mathrm{lazy}}10^{d-i } \cdot
     \left(\frac{2(d-i)}{n-i}\right)^{d-i}\|g\|_2 \\
  & = 1.001p_{\mathrm{lazy}} \|g\|_2  
    \sum_{i=0}^{d-1}   \left(\frac{20ed}{n-i}\right)^{d-i}
    \le .001p_{\mathrm{lazy}} \|g\|_2.  
\end{align*}
We deduce $\| T(\nu)g - p_{\mathrm{lazy}} g\| \le 0.01p_{\mathrm{lazy}}\|g\|_2$, as required.
\end{proof}

To complete the proof of our hypercontractivity theorem,
we prove these two technical lemmas in the final two subsections of this section.

\subsection{Staying coefficients of spread couplings}

Here we prove the first technical lemma, concerning the staying coefficients $\aA_S = \nu(A_S)$.

\begin{proof}[Proof of Lemma \ref{lem:Upper bound on alpha_S}]
Let $(a,b) \sim \nu$ be a $10/n$-spread left-invariant coupling on $[n]_d$ with $n \ge 10^5 d$.
For $S \sub [d]$ let $\aA_S = \nu(A_S)$ with $A_S := \{ \text{stay}(a,b)=S \}$.
For $S' \sub S$ we define events $E_{S'} = \{ \{i:a_i = b_i\} = S'\}$
and \[E'_{S'} = E_{S'} \cap \{ a_{S'}=b_{S'} = (n,\ldots ,n-|S'|+1) \}.\]
We note that $\nu(A_S|E_{S'})=\nu(A_S|E'_{S'})$ by $S_n$-invariance.
We bound $\aA_S$ using the decomposition
\begin{equation} \label{eq:aAS}
 \aA_S = \nu(A_S) = \sum_{S'}\nu(E_{S'})\nu(A_S|E_{S'})
 = \sum_{S'}\nu(E_{S'})\nu(A_S|E'_{S'}).
 \end{equation}
 
We write $\nu(A_S|E'_{S'}) = \nu_{S'}(A'_S)$, where
$(a',b') \sim \nu_{S'}$ is the distribution obtained from $(a,b) \sim \nu | E'_{S'}$
by deleting the coordinates $a_{S'}=b_{S'} = (n,\ldots ,n-|S'|+1)$,
and $A'_S = \{ \text{stay}(a',b') = S \sm S'\}$.
We note that $\nu_{S'}$ is left-invariant, 
so has uniform marginals (denoted $\mu$) on $[n-|S'|]_{d-|S'|}$.
For any $(a',b')$ obtained from $(a,b) \in E'_{S'}$, 
we have $a'_i \ne b'_i$ for all $i' \in [d] \sm S'$ by definition of $E_{S'}$,
so by spreadness of $\nu$,
\[  \nu(E'_{S'})\nu_{S'}(a',b') = \nu(a,b) \le (10/n)^{d-|S'|} \nu(a,a).  \] 
By $S_n$-invariance we have $\nu(E_{S'}) = \left|[n]_{|S'|}\right| \nu(E'_{S'})$
and $\nu(a,a) =  p_{\text{lazy}} / |[n]_d|$, so
\[  \nu(E_{S'})\nu_{S'}(a',b') \le (10/n)^{d-|S'|} p_{\text{lazy}} \mu(a').    \]
For each $a' \in [n-|S'|]_{d-|S'|}$ we have
$|\{b': (a',b') \in A'_S\}| \le  |S \sm S'|! n^{d-|S|}$, so
\[  \nu(E_{S'})\nu_{S'}(A'_S) \le 10^{d-|S'|} p_{\text{lazy}} (r/n)^r 
\text{ with }  r := |S \sm S'|.  \]
Substituting in \eqref{eq:aAS} we obtain
\begin{align*}
\aA_S & \le p_{\text{lazy}} \sum_{r=0}^{|S|}\binom{|S|}{r} 10^{d+r-|S|} (r/n)^r  \\
& \le  p_{\text{lazy}} 10^{d-|S|} \sum_{r \ge 0} (10e|S|/n)^r \le 1.001 p_{\text{lazy}} 10^{d-|S|}. \qedhere
\end{align*}
\end{proof}

\subsection{Norms of staying operators}

It remains to prove the second technical lemma, concerning the norms 
of the staying operators $T_S$ on $V_{=d}(L^2([n]_d))$

\subsubsection*{Restrictions}
First we make some preliminary remarks on restrictions.
For $R \sub [n]$ we write $R_d$ for the set of $d$-tuples in $R^d$ with distinct coordinates.
For $f \in L^2(R_d)$, $S \sub [d]$, $x \in R_{|S|}$, $X =\{x_1, \ldots , x_{|S|}\}$ we define 
$f_{S \to x} \in (R \sm X)_{d-|S|}$ by $f_{S \to x}(y)=f(x,y)$, with $x$ in the $S$ coordinates.
The following lemma shows that pure degree behaves nicely under restriction.

\begin{lem}\label{lem:restrictions and juntaness}
Let $f \in V_{=d}(L^2([n]_d))$, $S \sub [d]$, $x \in [n]_{|S|}$ and $X =\{x_1, \ldots , x_{|S|}\}$
Then $f_{S \to x} \in V_{=d-|S|}(L^2(([n] \sm X)_{d-|S|}))$.
\end{lem}
\begin{proof}
By definition, $f$ is orthogonal to all $(d-1)$-juntas
and we must show that $f_{S \to x}$ is orthogonal to all $(d-|S|-1)$-juntas.
We consider the operator $F: L^2([n]_d) \to L^2(([n] \sm X)_{d-|S|})$ defined by $Fh=h_{S \to x}$
and its adjoint $F^*$ given by $F^* g(x',y) = g(y) \cdot 1_{x'=x} |[n]_{|S|}|$.
If $g$ is a $(d-|S|-1)$-juntas then $F^* g$ is a $(d-1)$-junta,
so $\langle Ff, g \rangle = \langle f, F^*g \rangle = 0$.
\end{proof}

We also recall the disjointness operator from Section \ref{subsec:Kneser}:
for $2k\le |R|$ we let $\mc{D}_{R,k}$ be the operator on $R_k$ corresponding to the coupling
where $(x,y) \in R_k \times R_k$ is uniformly random 
subject to $\{x_1,\dots,x_k\} \cap \{y_1,\dots,y_k\} = \es$.

\begin{proof}[Proof of Lemma \ref{lem:s-staying kill spec_k}]
Let $f \in V_{=d}(L^2([n]_d))$ and  $T=T(\nu)$ be a left-invariant $S$-staying operator on $L^2([n]_d)$.
By definition $\text{stay}(a,b)=\{i: b_i \in \{a_1,\dots,a_d\}\}=S$ whenever $\nu(a,b)>0$.
We let $g=\frac{Tf}{\|Tf\|_2}$ (so $\|g\|_2=1$) and claim that
\[ \|Tf\|_2= \langle Tf, g\rangle = \mb{E}_{(a,b) \sim \nu} [f(b)g(a)] 
  = \mb{E}_{b_S,S',a_{S'}} [\langle \mc{D}(f_{S \to b_S}), g_{S' \to a_{S'}}\rangle],  \]
where $\{b_i: i \in S\} = \{ a_i: i \in S' \} = X$ and $\mc{D} = \mc{D}_{[n] \sm X, d-|S|}$.
Indeed, this claim holds as conditional on $b_S,S',a_{S'}$, by left-invariance
we have $a_{S'^c},b_{S^c}$ in $([n] \sm X)_{d-|S|}$
uniformly random subject to $\{a_i: i \in S'^c\} \cap \{b_i: i \in S^c\} = \es$.
 
Now $\mb{E}_{S',a_{S'}} \| g_{S' \to a_{S'}} \|_2^2 = \|g\|_2^2=1$,
so applying Cauchy--Schwarz twice we obtain 
    \begin{equation}\label{eq:cs1}
    \|Tf\|_2 \le  \sqrt{\mathbb{E}_{b_S}\|\mc{D}[f_{S\to b_S}]\|_2^2}.
    \end{equation}
   
 By Lemma \ref{lem:restrictions and juntaness}
 we have $f_{S\to b_S} \in V_{=d-|S|}(L^2(([n]\sm X)_{d-|S|}))$,
 so by Lemma \ref{lem:Kneser-type operator} we have 
 $\|\mc{D}[f_{S\to b_S}]\|_2 \le \left( \frac{2(d-|S|)}{n-|S|}\right)^{d-|S|}\|f_{S\to b_S}\|_2.$ 
 Substituting in \eqref{eq:cs1} we conclude 
 \begin{align*}
  \|Tf\|_2  & \le \left( \frac{2(d-|S|)}{n-|S|}\right)^{d-|S|} \sqrt{\mathbb{E}_{b_S}\|f_{S\to b_S}\|_2^2} 
  = \left( \frac{2(d-|S|)}{n-|S|}\right)^{d-|S|} \|f\|_2. \qedhere
 \end{align*}
\end{proof}

\section{Level and spectral estimates}

In this section we prove our level $d$ inequality
and spectral estimates for global convolution operators.
We also give an example showing sharpness. 
We will prove the following general form of our level $d$ inequality
(Theorem \ref{thm:level-d for global functions_intro} is a special case).

\begin{thm}\label{thm:level-d for biglobal functions}
Let $f \in L^2(S_n)$ be $(r,\gG_1,\gG_2,d)$-biglobal with $r>1$, $\gG_2>\gG_1>0$ and
$d\le\min\left(\tfrac{1}{4}\log(\gG_2/\gG_1), 10^{-5}n \right)$.
Then  $\| f^{=d} \|_2^2\le \gG_1^2 \left( 10^6 r^2 d^{-1} \log (\gG_2/\gG_1) \right)^d$.

Furthermore, there is an absolute constant $C>0$ such 
if $f$ is $r$-biglobal with $\log(\|f\|_2/\|f\|_1) < 10^{-6} n,$ then
\[\| f^{=d} \|_2^2 \le \|f_1\|^2 \left( Cr^2 \log(e\|f\|_2/\|f\|_1) \right)^d\] for any $d$.
 \end{thm}

It will be helpful later to distinguish between 
the strict level $n-\lL_1$ and level $\min(n-\lL_1,n-\lL'_1)$ of $\lL \vert n$,
recalling that $\lL'$ denotes the conjugate partition
and multiplication by the sign character interchanges $V_\lL$ and $V_{\lL'}$.
We recall the degree decomposition $L^2(S_n) = \bigoplus_d V_{=d}$
and that $V_\lL \le V_{=d}$ if $\lL$ has strict level $d$.
We also consider the following (non-canonical) level decomposition.

\begin{defn}
For $d<n/10$ we write $W_{=d} = V_{=d} \oplus \text{sign} \cdot V_{=d}$
for the sum of all $V_\lL$ of level $d$, and consider the \emph{level decomposition}
$L^2(S_n) = W \oplus \bigoplus_{i=0}^{n/10} W_{=d}$,
where $W$ is the sum of all $V_\lL$ with level $>n/10$.
We let $P_d$ and $P_{>d}$ denote the projections on $W_{=d}$ and 
 $W_{>d} :=  W \oplus \bigoplus_{i=d+1}^{n/10}W_{=i}$.
We write $\tilde{f}$ for $f\cdot \mathrm{sign}.$ 
\end{defn}

\begin{rem} \label{rem:Pd}
For any $f \in L^2(S_n)$,
we note that $\tilde{f}$ has the same (bi)globalness properties as $f$,
as multiplying by sign does not affect norms.
Thus multiplying by a factor $4$ the bound 
on $\| f^{=d} \|_2^2$ in Theorem \ref{thm:level-d for biglobal functions}
gives a bound on $\| P_d f \|_2^2 $.
\end{rem}

Now we state our spectral bound on convolution operators,
noting that it implies Theorem \ref{thm:Eigenvalues of operators_intro}.

\begin{thm} \label{thm:Eigenvalues of operators}  
Let $f$ be an $r$-biglobal function. Let $T_f$ be the operator $g\mapsto f * g$.  When 
$d\le \min(\frac{1}{4}\log\left(\frac{\|f\|_2}{\|f\|_1}, 10^{-5} n \right)$ we have 
\[ \|T_f\|_{W_{=d}} \le  \left\|f\right\|_1\left( 10^7 r^2 n^{-1} \log(\|f\|_2/\|f\|_1)\right)^{d/2}.\] 
We also have  $\|T_f\|_{W_{>d}} \le \| f\|_2 (ed/n)^{d/2}$ for all $d$.
\end{thm}

We start by deducing Theorem \ref{thm:Eigenvalues of operators} 
from Theorem \ref{thm:level-d for biglobal functions} 
and the following lemma implicit in \cite{keevash2022largest}
(see the proof of Theorem 2.8 therein.) 

\begin{lem}\label{lem:Other KLM}
For any $f \in L^2(S_n)$ and $\lL \vert n$ we have
$\dim(V_\lL) \|T_f\|_{V_\lL}^2 \le  \|f^{=\lL} \|_2^2$.
\end{lem}

\begin{proof}[Proof of Theorem \ref{thm:Eigenvalues of operators}]
The second bound follows from Lemma \ref{lem:Other KLM},
using $\| P_{>d} f \|_2 \le \|f\|_2$ and $\dim(V_\lL) \ge \left( \tfrac{n}{ed} \right)^d$ 
for any $\lL$ of level $d$ (see \cite[Claim 1]{ellis2011intersecting}).
The first bound follows similarly, 
also using $\|P_d f\|_2\le \|f^{=d}\|_2 + \|\tilde{f}^{=d}\|_2$
and applying Theorem \ref{thm:level-d for biglobal functions}
to $f$ and $\tilde{f}$ (see Remark \ref{rem:Pd}).
\end{proof}

It remains to prove the level $d$ inequality in $S_n$.
Similarly to our proof of hypercontractivity, we will deduce it via the coupling 
from the sharp level $d$ inequality in $[n]^n$. 
We start with the statement that follows most easily:
we consider the projection $P_d$ on $L^2([n]^n)$ given by $f\mapsto f^{ \le d}$
and deduce a spectral bound for the operator on $L^2(S_n)$ obtained from $P_d$
in the same way that we obtained our noise operator on $L^2(S_n)$ 
from the noise operator on $L^2([n]^n)$, namely
 \[T_d := \mathbb{E}_{\sigma\sim S_n}[R_{\sigma^{-1}}T_C^*P_d T_C R_\sigma].\] 

\begin{lem}\label{lem:hypercontractivity for T_d}
Let $f \in L^2(S_n)$ be $(r,\gG_1,\gG_2,d)$-biglobal 
with $r>1$, $\gG_2>\gG_1>0$ and $2d \le \log(\gG_2/\gG_1)$.
Then $\|T_d f\|_2^2\le \gG_1^2 ( 10^5 r^2 d^{-1} \log(\gG_2 /\gG_1) )^d$. 
\end{lem}

\begin{proof}
For each $\sS \in S_n$, clearly $R_\sS f$ has the same biglobalness as $f$.
This is also true of  $T_C R_\sS f$ by  Lemma \ref{lem:biglobalness is preserved under the coupling}.
Applying Theorem \ref{thm:KLM2} gives
  \[\|P_dT_C R_{\sigma}f\|=\|(T_C R_\sigma f)^{\le d}\|_2^2
  \le \sum_{i=0}^d \gG_1^2 ( 2200 r^2 i^{-1} \log(\gG_2 /\gG_1) )^i. \]
As $T_C^*$ is contractive we deduce the same bound for
$\| R_{\sigma^{-1}}T_C^*P_d T_C R_\sigma f \|$.
Using  Cauchy--Schwartz with respect to averaging over $\sS$
we deduce the same bound for $\|T_d f\|$.
The lemma follows as each summand is at least twice the previous.
\end{proof}

Now we can deduce the level $d$ inequality.

\begin{proof}[Proof of Theorem \ref{thm:level-d for biglobal functions}]
As $P_d$ is self-adjoint and$P_d^2 =P_d$ we have
\[\langle T_d f, f \rangle = \langle \mb{E}_{\sigma}R_{\sigma}^*T_C^*P_d^* P_d T_C R_\sigma f, f \rangle
 = \mathbb{E}_\sigma[\|(T_C R_{\sigma} f^{=d}\|_2^2.]\]
Recall from \eqref{eq:4.1} that if $n \ge 10^5 d$ and $g$ is a $d$-junta then
$\|(T_{C}R_{\tau}g)^{\le d}\|_{2} \ge 3^{-d} \|g\|_{2}$,
so $\langle T_d g, g \rangle \ge 3^{-d} \|g\|_{2}$.
Similarly to the proof of Lemma \ref{lem:Part 3 of hypercontractivity},
we see that $T_d$ commutes with the action of $S_n$ from both sides,
so as in the proof of Theorem \ref{thm:hypercontractivity for global functions in the symmetric group}
we deduce $\langle T_d g, g \rangle \ge 3^{-d} \|g\|_{2}$ for all $g \in V_{=d}$.

Now for any $f$, as $T_d$ is a self-adjoint $S_n$-morphism we have
\[
\|T_d f\|_2^2 \ge  \| (T_d f)^{=d} \|_2^2 = \|T_d (f^{=d})\|_2^2 \ge 3^{-d}\|f^{=d}\|_2^2. 
\]
The first statement of the theorem follows by combining this lower bound on $\| T_d f\|$
with the upper bound in Lemma \ref{lem:hypercontractivity for T_d}.
The second statement follows from the first,
or from the trivial bound $\| f^{=d} \|_2 \le \| f\|_2$ if $d > \tfrac{1}{4}\log(\|f_2\|/\|f_1\|)$.
\end{proof}

We conclude this section with an example showing
that the bound in our level $d$ inequality is essentially optimal.

\begin{ex} \label{ex:leveld}
Let $s>0$, fix $S \sub [n]$ of size $s$ and let $f=1_A$, where $A = \{ \sS \in S_n: \sS(S) \sub [n/2] \}$.
\end{ex}

\begin{lem}
For integers $d,s,n$ with $s$ sufficiently large with respect to $d$ and $n$ sufficiently large with respect to both we have $\mu := \mu(A) = 2^{-s + O(s^2/n)}$ and 
$\| f^{=d} \|_2^2 > \mu^2 (cd^{-1}\log(1/\mu))^d$ for some absolute constant $c>0$.
\end{lem}
\begin{proof}
Write $g = 2|S|^{-1/2}\left(\sum_{i\in S,j\in [n/2]}x_{i\to j} - \frac{|S|}{2}\right).$ Then as $n$ tends to infinity the distribution of the value of $g$ on a uniformly random permutation converges to the distribution of a normal $X\sim N(0,1)$ random variable. Let $h= g^d.$ Then we have 
\[\|h\|_2^2 = \|g\|_{2d}^{2d} = \|X\|_{2d}^{2d}+o(1) = (2d-1)!! + o(1),\]
where the $o(1)$ is with respect to $n$ tending to infinity. 
On the other hand, we have 
$\langle f,h \rangle = \mu(A)|S|^{d/2}.$ Since $h$ is of degree $d$ we have by Cauchy--Schwarz 
\[
\|f^{\le d}\|_2 \ge \left\langle f^{\le d}, \frac{h}{\|h\|_2} \right\rangle = \left\langle f, \frac{h}{\|h\|_2} \right\rangle \ge \mu(A) |S|^{d/2}\frac{1}{\sqrt{(2d-1)!!}} + o(1). 
\]
It follows that there exists an absolute constant $c>0$ such that \[\|f^{\le d}\|_2^2\ge 2\mu^2 (cd^{-1}\log(1/\mu))^d.\]
Since $f$ is $4$-biglobal, we may apply Theorem \ref{thm:level-d for biglobal functions} to deduce that there exists an absolute constant $C>0$ for which we have \[\|f^{\le d-1}\|_2^2 \le \mu^2 (C(d-1)^{-1}\log(1/\mu))^{d-1}.\] Provided that $s$ is sufficiently large we therefore have 
\begin{align*}
\|f^{=d}\|_2^2 & = 
\|f^{\le d}\|_2^2 - \|f^{\le d-1}\|_2^2 \ge \mu^2 (cd^{-1}\log(1/\mu))^d.
\qedhere
\end{align*}
\end{proof}

\newpage

\part{Applications}
    
In this part of the paper we apply the theory developed in the previous part
to the applications mentioned in the introduction. 
The next section contains our results on approximate subgroups in $A_n$
(towards the Helfgott-Lindenstrauss nonabelian analogue 
of the Polynomial Freiman-Ruzsa Conjecture).
In Section \ref{sec:mixbog} we present our results
on general mixing and Bogolyubov in $A_n$;
these are applied in Section \ref{sec:diam}
to our tight bounds for the diameter problem.
In Section \ref{sec:prodmix} we present our results on product mixing,
which we apply to analogues of the Waring problem in Section \ref{sec:waring}
and Roth's Theorem in Section \ref{sec:roth}.
We conclude in Section \ref{sec:Sharpness} with an example
showing sharpness of our results on mixing.

\section{Towards Polynomial Freiman-Ruzsa in $A_n$} \label{sec:PFR}

In this section we prove our covering result for $K$-approximate subgroups 
$A$ in $S_n$ with $\mu(A) = e^{-O(n^{1-\eps})}$, showing that
they can be efficiently covered by cosets of a large subgroup, 
with parameters depending polynomially on $\mu(A)$ and $K$.

Our first lemma shows that any set has a suitable global restriction.

\begin{lem}\label{lem: the existence of a global restriction}
Suppose $A \sub S_n$ with $\mu(A) \ge r^{-\ell}$.
Then there is an $r$-global restriction $A_{I\to J}$ with $|I|=|J| \le \ell$
such that $\mu(A_{I\to J}) = \max_K \mu(A_{I \to K}) \ge \mu(A) r^{|I|}$.
\end{lem}

\begin{proof}
If $A$ is $r$-global then we are done. 
Otherwise, we can find a restriction $I \to J$ with
$\mu(A_{I\to J}) = \max_K \mu(A_{I \to K}) \ge \mu(A) r^{|I|}$.
Repeating this process, we find the required restriction.
\end{proof}

Our next lemma is a growth result for global sets. 
We adopt a somewhat technical quantitative formulation 
so that we can cover all densities that are $e^{-O(n)}$,
but we also point out a more easily digested consequence:
if $A$ is $O(1)$-global with $e^{-c\sqrt{n}} < \aA < c$
then $\mu(A^2) \ge 1 - O(c)$.

\begin{lem}
\label{lem:growth of AB}
There are constants $c,C>0$ so that if $A \sub S_n$ is $r$-global with $r>1$
and $\mu(A) = \aA \in (e^{-cn/r^2},c)$ then $\mu(A^2) \ge 1/(1+c+4M \cdot 2^{2M})^2$,
where $M :=  C^2 r^4\log^2(1/\alpha)/n$.
 \end{lem}

\begin{proof}
We write $f=g=\frac{1_{A}}{\mu(A)}$, $h=1_{A^2}$ and consider the inequality
\[ 1 = \left\langle f*g,h\right\rangle \le \|f*g\|_2 \|h\|_2 = \|T_f g\|_2 \sqrt{\mu(A^2)}.\]
We let $\ell=\lceil \tfrac18 \log(1/\aA) \rceil$ and estimate
\[\| T_f g\|_2 \le \| T_f \|_{W^{>\ell}} \| P_{>\ell} g \|_2 + \sum_{d=0}^{\ell} \| T_f \|_{W^{=d}} \| P_{=d} g \|_2.\]
By globalness we can apply 
Theorem \ref{thm:Eigenvalues of operators} (the spectral bound) to $T_f$ and 
Theorem \ref{thm:level-d for biglobal functions} (the level $d$ inequality) to $g$, 
obtaining an absolute constant $C$ with
\begin{align*}
& \sum_{d=0}^{\ell} \|T_f\|_{W^{=d}} \| P_{=d} g \|_2 
\le \sum_{d=0}^{\ell} X_d \text{ with } 
X_d := \left( \frac{C r^2 \log(1/\alpha)}{\sqrt{dn}} \right)^d = (M/d)^{d/2},  \\  
& \text{ and } \
 \| T_f \|_{W^{>\ell}} \| P_{>\ell} g \|_2 \le (e\ell/n)^{\ell/2} \|f\|_2 \|g\|_2 \le \aA^{\log(1/c)/2-2} \le c, 
\end{align*}
for $c$ sufficiently small. Noting that $X_{d+1} \le X_d/2$ for $d\ge 4M$
and $X_{d+1} \ge X_d$ whenever $M \ge 2d \ge 2$,
we deduce $\sum_{d=0}^{\ell} X_d \le 1 + 4M \cdot 2^{2M}$.

Thus $\mu(A^2) \ge \|T_f g\|_2^{-2} \ge 1/(1+c+4M \cdot 2^{2M})^2$.
\end{proof}

We conclude this section by proving our result on approximate subgroups in $S_n$. 

\begin{proof}[Proof of Theorem \ref{thm:approximate subgroups_intro}]
Let $A\subseteq S_n$ be a $K$-approximate group,
i.e.\  if $\text{id} \in A =A^{-1}$  and $A^2 \subseteq XA$ for some $X \sub G$ with $|X|=K$.
Suppose $\mu(A) \ge e^{-cn^{1-\eps}}$ and $\eps>1/\log n$.
We need to find  a subgroup $H$ with $\mu(H) \ge \mu(A)^{C/\eps}$
such that $A \sub ZH$ for some $Z \sub A$ with $|Z| \le K^C |A|/|H|$.
We may assume $\mu(A) \le c$ is sufficiently small. 
Set $r=n^{\eps/5}$. By Lemma \ref{lem: the existence of a global restriction} 
there is an $r$-global restriction $A_{I\to J}$ with
$\mu(A_{I\to J}) = \max_K \mu(A_{I \to K}) \ge \mu(A) r^{|I|}$.
We let $H := U_I$. As $\mu(A_{I\to J}) \le 1$ we have
$\mu(H) \ge n^{-|I|} = r^{-5|I|/\eps} \ge \mu(A)^{5/\eps}$.
 
The main step of the proof is the following claim.

\begin{claim}
$\mu(A_{I\to J})\ge K^{-5}$.
\end{claim}

Indeed, assuming the claim, we can complete the proof 
with the following standard Ruzsa covering argument.
Let $A' = \{ \sS \in A: \sS(I)=J \}$ and consider $Z \sub A$ 
maximal subject to $\{ zA': z \in Z \}$ all disjoint.
Then $A \sub \{ zH: z \in Z \}$ by maximality,
and $|Z||A'| \le |A|^2 \le K|A|$, so $|Z| \le K^6 |A|/|H|$.

To prove the claim, fix any  $\sigma \in A_{I\to J}$ 
and consider $B := A_{I\to J}\sigma^{-1} \sub A^2$. 
Then $B \sub U_I \cong S_{[n]\sm I}$ is $r$-global with
$\mu(B_{I \to I}) = \mu(A_{I\to J}) \ge \mu(A)  \ge e^{-cn^{1-\eps}}$.

We will show that $\mu(B_{I \to I}^2) \ge \mu(A_{I \to J})^{0.1}$.
Indeed, this clearly holds if $\mu(B_{I \to I}) \ge \mu(A_{I \to J})^{0.1}$.
Otherwise, as we can assume $\mu(A_{I \to J})$ is small,
by Lemma \ref{lem:growth of AB} we have
\[\mu(B_{I \to I}^2) \ge \frac{1}{(1+c+4M \cdot 2^{2M})^2},\]
where \[M \le  C^2 r^4\log^2(1/\mu(A_{I \to J}))/n \le cC^2 \log(1/\mu(A_{I \to J})),\]
using $r=n^{\eps/5}$, so $\mu(B_{I \to I}^2) \ge  \mu(A_{I \to J})^{0.1}$ for $c$ small.

Next we fix a set $Y$ of representatives for all non-empty $A_{I \to K}$.
As $\mu(A_{I\to J}) = \max_K \mu(A_{I \to K})$ 
we have $\mu(A) \le |Y| \mu(A_{I\to J})\mu(H)$.

On the other hand,  $YB^2 \sub A^5 \sub X^4 A$, so 
\[ |Y|  \mu(A_{I \to J})^{0.1}  \mu(H) \le K^4 \mu(A)  \le K^4|Y| \mu(A_{I\to J})\mu(H). \] 
We deduce $K^4 \mu(A_{I\to J})^{0.9} \ge 1$, which implies the claim, and so the theorem.  
\end{proof}

\section{Mixing and Bogolyubov} \label{sec:mixbog}

In this section we prove our results on mixing times for global sets
and our Bogolyubov lemma for the symmetric group.
We start by stating a Bogolyubov lemma for global sets and showing that 
it implies our general Bogolyubov result stated in the introduction (Theorem \ref{thm:Bogolyubov_intro}).
Recalling the example $S_{n/2}\times S_{n/2}$ mentioned earlier,
we cannot get analogous results for global sets below density $2^{-n}$,
as the Cayley graph may not even be connected.

\begin{thm}\label{cor:global Bogolyubov}
There is $c>0$ so that if $M<2\log n$ is an integer and $A\sub A_n$ is $r$-global
with $\mu(A) \ge e^{-cr^{-2}n^{1-1/M}}$ then $|A^M|\ge 0.99|A_n|$ and $A^{2M}=A_n$. 
\end{thm}

Now we show that this implies our general Bogolyubov lemma,
which finds a polynomially dense subgroup inside a constant power of our set.

\begin{proof}[Proof of Theorem \ref{thm:Bogolyubov_intro}]
Consider $A \sub A_n$ with $A=A^{-1}$,  $\mu(A) \ge e^{-cn^{1-1/M}}$ and $M < \log n$.
We apply Lemma \ref{lem: the existence of a global restriction} with $r=n^{1/(5M)}$
to find an $r$-global restriction  $A_{I\to J}$ with 
$\mu(A_{I\to J}) \ge \mu(A) r^{|I|} \ge e^{-cr^{-2}(n-|I|)^{1-1/(2M)}}$.
We fix any $\sS \in A_{I \to J}$ and consider $B = A_{I \to J} \sS^{-1} \sub (A^2)_{I \to I}$,
which is an $r$-global subset of $U_I \cong A_{[n] \sm I}$.
By Theorem \ref{cor:global Bogolyubov} we have 
$B^{4M} = A_{[n] \sm I}$, so $U_I \sub A^{8M}$,
where 
\begin{align*}
    \mu(U_I) & \ge n^{-|I|} = r^{-5M|I|} \ge \mu(A)^{5M}. \qedhere
\end{align*}
\end{proof}

We now consider $L^2$-mixing times for general functions on $S_n$.
We state the main result of the section and show that it implies  Theorem \ref{cor:global Bogolyubov}
and our mixing time result for $A_n$ stated in the introduction (Theorem \ref{thm:mixing time_intro}).

\begin{thm}\label{thm:mixing time}
For each $\eps>0$ there exists $c>0$ so that
if $M<2\log n$ is an integer and $f\in L^2(S_n)$ is $r$-biglobal
with $\|f\|_1=1$ and $\|f\|_2 \le e^{cr^{-2}n^{1-1/M}}$
then $\|f^{*M}-f_{=0}^{*M}\|_2 < \eps$.  
\end{thm}

\begin{proof}[Proof of Theorem \ref{cor:global Bogolyubov}]
Writing $f=1_A/\mu(A)$, applying Theorem \ref{thm:mixing time}
with $\eps = 0.01$ and $c>0$ sufficiently small, we have 
\[\mu(A_n \sm A^M) \le \|f^{*M}-f_{=0}^{*M}\|_1 \le \|f^{*M}-f_{=0}^{*M}\|_2 < 0.01,\]
i.e.\ $\mu(A^M) > 0.99$. Now for each $x \in A_n$ we have
$\mu(xA^{-M}) = \mu(A^M) = 0.99$, so $A^M \cap xA^{-M} \ne \es$, i.e.\ $x\in A^{2M}$.
\end{proof}

\begin{proof}[Proof of Theorem \ref{thm:mixing time_intro}]
For $f \in L^2(A_n)$  we define $f^\ua \in  L^2(S_n)$
by $f^\ua(\sS) = 2f(\sS)$ for $\sS \in A_n$ or $f^\ua(\sS)=0$ otherwise.
We note that $(f^\ua)^{=0} = \|f\|_1 (1+\mathrm{sign})$ and $(f^{*M})^\ua = (f^\ua)^{*M}$.
Now let $f \in L^2(A_n)$ be $r$-biglobal with $\|f\|_1=1$ and $\|f\|_2 \le e^{cr^{-2}n^{1-1/M}}$.
Then $f^\ua \in L^2(S_n)$ is $2r$-biglobal with $\|f^\ua\|_1=\|f\|_1$ and $\|f^\ua\|_2 = \sqrt{2} \|f\|_2$.
Now $\|f^{*M} - 1\|_2 \le \|(f^\ua)^{*M} - 1 -\mathrm{sign}\|_2 <\epsilon$ by Theorem \ref{thm:mixing time}.
\end{proof}

It remains to prove our general mixing time result.

\begin{proof}[Proof of Theorem \ref{thm:mixing time}]
Suppose $f\in L^2(S_n)$ is $r$-biglobal
with $\|f\|_1=1$ and $\|f\|_2 \le e^{cr^{-2}n^{1-1/M}}$, 
with $0 < c < c(\eps)$ small.
Let $\ell =\tfrac18 \log\|f\|_2$. As $T_f$ preserves each $W_{=d}$ we have 
$\| P_{=d}(f^{*M}) \|_2 =    \| T_f^{(M-1)} P_{=d}(f) \|_2 \le \| T_f \|_{W_{=d}}^{M-1} \|P_{=d} f\|_2$, so
\begin{align*} 
  \|f^{*M}-f_{=0}^{*M}\|_2 & \le \| P_{>\ell}(f^{*M}) \|_2 + \sum_{d=1}^\ell \| P_{=d}(f^{*M}) \|_2  \\
& \le  \| T_f \|_{W_{>\ell}}^{M-1} \| P_{>\ell}(f) \|_2 + \sum_{d=1}^\ell   \| T_f \|_{W_{=d}}^{M-1} \|P_{=d} f\|_2. 
\end{align*}
By globalness we can apply Theorems \ref{thm:level-d for biglobal functions} 
and  \ref{thm:Eigenvalues of operators}; for $d \le \ell$ we have
\[  \| T_f \|_{W_{=d}}^{M-1} \|P_{=d} f\|_2 \le (Cr^2 n^{-1} \log \|f\|_2)^{(M-1)d/2} \cdot (Cr^2\log \|f\|_2)^{d/2} < \eps \cdot 4^{-d} \]
as $C$ is a constant and $c<c_0(\eps)$, 
and as  $\ell =\tfrac18 \log\|f\|_2$ and $e\ell/n < c$ also
\[  \| T_f \|_{W_{>\ell}}^{M-1} \| P_{>\ell}(f) \|_2  \le \|f\|_2^M (e\ell/n)^{(M-1)\ell/2} < \eps/4. \]
Combining these inequalities we deduce the theorem.
\end{proof}

\section{The diameter of Cayley graphs} \label{sec:diam}

Recall that the diameter  of $\text{Cay}(A_n,A)$ with $A^{-1} = A \sub A_n$,
also known as the covering number $\mathrm{cn}(A)$ of $A$, 
is the minimal $m$ for which $A^m=A_n.$ 
In this section we will prove Theorem \ref{thm:diameter bound},
which gives an essentially sharp bound on $\mathrm{cn}(A)$ when $\mu(A) < e^{-O(n)}$.
Our strategy will be to apply our Bogolyubov result to find a large subgroup $U_I$ in a power of $A$,
then follow an iterative process that finds larger subgroups until we hit $A_n$.

Our first lemma is the following estimate on the diameter
of the Schreier graph $\mathrm{Sch}(A_n,A,[n]_{\ell})$,
i.e.\ the graph on the set $[n]_\ell$ of $\ell$-tuples in $[n]^\ell$ with distinct coordinates 
where for each $I \in [n]_\ell$ and $\sS \in A$ there is an edge $\{I,\sS(I)\}$.

\begin{lem}\label{schreier}
There are $C=C_{\ref{schreier}}$, $c=c_{\ref{schreier}}$
so that if $A\sub A_n$ is $r$-global with $r>1$
and $\mu(A) \ge e^{-cr^{-2}n^{1-\zeta}}$ 
then the diameter of the Schreier graph $\mathrm{Sch}(A_n, A, X_{\ell})$ 
is at most $C \left\lceil\frac{\ell\log r}{ \zeta \log n}\right\rceil.$
\end{lem}

\begin{proof}
Note that $\mathrm{Sch}(A_n,A,[n]_{\ell})$ is a quotient of $\mathrm{Cay}(A_n,A)$,
where each $I \in [n]_\ell$ is the projection of $\{ \sS \in A_n: \sS([\ell]) = I \}$.
Thus the second eigenvalue $\lL_2$ of $\mathrm{Sch}(A_n,A,[n]_{\ell})$ is bounded by that of $\mathrm{Cay}(A_n,A)$,
so $\lL_2 \le n^{-\zeta}$ by Theorem \ref{thm:Eigenvalues of operators}.
Now we fix any $I,I'$ in $[n]_\ell$ and bound the distance from $I$ to $I'$ in $\mathrm{Sch}(A_n,A,[n]_{\ell})$.
Writing $N(I)$ for the neighbourhood of $I$, by globalness we have
\[\mu(A) = \sum_{J \in N(I)} \mu(A_{I \to J}) \mu(U_I) \le \mu(N(I)) r^\ell \mu(A),\]
so $\mu(N(I)) \ge r^{-\ell}$. Now by expansion, 
the distance from $I'$ to $N(I)$ is at most $O( \log(\mu(N(I)))/\log(\lL_2) )$.
\end{proof}

We use the previous lemma to prove our next lemma,
which will be used iteratively to construct larger subgroups.

\begin{lem}\label{lem: eating I}
Suppose $A\sub A_n$ is symmetric and generating,
with $\mu(A_{I \to I}) \ge e^{-c_{\ref{schreier}}(n-|I|)^{1-\eps}}$ 
and $A_{I \to I}$ is $n^{\eps/5}$-global, with $|I|<n/10$.
Let $\es \ne I' \sub I$. Then there exists
$U_{I''} \sub A^{2C_{\ref{schreier}}|I|} U_{I'} A^{2C_{\ref{schreier}}|I|}$ with $|I''|<|I'|$.
\end{lem}

\begin{proof}
Fix any $x \in I'$. By connectivity of the Schreier graph $\mathrm{Sch}(A_n,A,[n])$
we can find $\sS \in A^{|I|}$ with $\sS(x) \in [n] \sm I$. 
Let $J = \sS(I')$ and note that $J \sm I \ne \es$ and $U_J = \sS U_{I'} \sS^{-1}$.
We fix any ordering on $J \sm I$ to make it a tuple. 
For any $J_1,J_2 \in ([n] \sm I)_{|J \sm I|}$,
we apply Lemma \ref{schreier} to $A_{I \to I} \sub A_{[n] \sm I}$ with $\zeta=\eps/2$,
noting that $\frac{\log (n^{\eps/5})}{\zeta \log (n-|I|)} < 1$,
obtaining $\sS_1,\sS_2 \in (A_{I \to I})^{C_{\ref{schreier}}|J \sm I|}$
with $\sS_1(J \sm I)=J_1$, $\sS_1(J \sm I)=J_2$.
We note that $\sS_1 U_J \sS_2^{-1} = U_{I'' J_2 \to I'' J_1}$, where $I'' := J \cap I$.
Taking the union over $J_1,J_2$, we obtain 
$U_{I''} \sub A^{|I| + C_{\ref{schreier}}|J \sm I|} U_J A^{|I| + C_{\ref{schreier}}|J \sm I|}$.
\end{proof}

We now combine the above lemmas and our Bogolyubov result to prove the diameter bound.

\begin{proof}[Proof of Theorem \ref{thm:diameter bound}]
Let $A\subseteq A_n$ with $A=A^{-1}$ and $\mu(A) \ge e^{-cn^{1-1/M}}$,
where $M = 1/\eps < 1/\log n$ and we may assume $M$ is an integer. 
By Theorem \ref{thm:Bogolyubov_intro} and its proof
we find $U_I \sub A^{8M}$ with $\mu(U_I) \ge n^{-|I|} \ge \mu(A)^{5M}$,
and also $A_1 \sub A^2$ such that $B = (A_1)_{I \to I}$
is $n^{\eps/5}$-global in $U_I \cong A_{[n] \sm I}$ with $\mu(B) \ge \mu(A)$.
As $n^{-|I|} \ge \mu(A)^{5M}$ we have $|I| \le 5M\log_n(1/\mu(A)))$.
Now we apply Lemma \ref{lem: eating I} repeatedly
to construct a sequence $I = I_0, I_1, \dots, I_t = \es$,
where at step $i$ we have $I' = I_i$ and construct $I'' = I_{i+1}$.
We deduce $A^{8M + C|I|^2} = A_n$ for some absolute constant $C$.
\end{proof}

\section{Product mixing} \label{sec:prodmix}

In this section we prove the following stronger variant 
of Theorem \ref{thm:The global mixing property_Intro}.  

\begin{thm}\label{thm:product mixing}
There are absolute constants $c,n_0>0$ so that if $n>n_0$
and $f,g,h \in L^2(S_n)$ are $r$-biglobal with $r>1$ and
$\frac{\|f\|_2}{\|f\|_1},\frac{\|g\|_2}{\|g\|_1}, \frac{\|h\|_2}{\|h\|_1} \le e^{cn^{1/3}r^{-2}}$ then 
\[\left|\langle f*g, h \rangle -  \mathbb{E}[f]\mathbb{E}[g]\mathbb{E}[h] -
\langle f, \mathrm{sign} \rangle \cdot \langle g, \mathrm{sign}\rangle \cdot \langle h, \mathrm{sign}\rangle \right| \le  0.01\|f\|_1\|g\|_1\|h\|_1.\]
\end{thm}

Next we show that Theorem \ref{thm:product mixing} implies Theorem \ref{thm:The global mixing property_Intro}. 

\begin{proof}[Proof of Theorem \ref{thm:The global mixing property_Intro}]
Recalling Definition \ref{def:prodmix},  we consider  $A,B,C\subseteq A_n$
that are all $100$-global of density at least $e^{-cn^{1/3}}$
and show that $(A,B,C)$ is $0.01$-mixing for products,
i.e.\  $\Pr[a\in A,b\in B,ab\in C] \in (0.99,1.01) \mu(A)\mu(B)\mu(C)$,
where $\mu(A) = |A|/|A_n| = 2|A|/n!$, etc and
$a,b\sim A_n$ are independent and uniform.

We consider the corresponding characteristic functions on $S_n$ given by
$f=\frac{n!1_A}{|A|}$, $g= \frac{n!1_B}{|B|}$ and $h= \frac{n!1_C}{|C|}$,
which are $200$-global, with $\|f\|_1=\|g\|_1=\|h\|_1=1$ and $\|f\|_2 =  \sqrt{n!/|A|}$, etc.
Also, $\langle f, \mathrm{sign} \rangle = \langle f,1 \rangle =  \mb{E}f = 1$ etc, and 
\begin{align*}
& \Pr_{a,b\sim S_n}[a\in A, b\in B,ab\in C] = \langle f*g, h \rangle |A||B||C|/n!^3, \text{ so } \\
& \Pr_{a,b\sim A_n}[a\in A, b\in B,ab\in C] = \langle f*g, h \rangle \mu(A)\mu(B)\mu(C)/2.
\end{align*}
By Theorem \ref{thm:product mixing} we have $|\langle f*g, h\rangle -2|<0.01$,
which gives the required estimate for $\Pr[a\in A,b\in B,ab\in C]$.
 \end{proof}
 
To complete the proof of Theorem \ref{thm:The global mixing property_Intro} we prove Theorem \ref{thm:product mixing}.

\begin{proof}[Proof of Theorem \ref{thm:product mixing}]
 Without loss of generality we have $\|f\|_1=\|g\|_1=\|h\|_1=1$. 
Our goal is to estimate $\langle f*g,h\rangle = \langle T_f g,h \rangle$. 

We let $\ell=\lceil \tfrac18 \log\|f\|_2 \rceil$ and use $T_f$-invariance of each $W_{=d}$ to expand as
\[\langle f*g,h\rangle =\sum_{d=0}^{\ell}\langle T_f P_d g,P_d h\rangle + \langle T_f P_{>\ell}g P_{>\ell}h\rangle.\]

Next we consider the main term in the expansion, which comes from $d=0$.
We have $P_0(g)=\mathbb{E}[g] + \langle g,\mathrm{sign}\rangle \mathrm{sign}$ and similarly for $P_0h$.
As $f*1=\mathbb{E}[f]$ and 
$f*\mathrm{sign}(x) = \mathbb{E}_y[f(y)\mathrm{sign}(xy^{-1})] = \langle f, \mathrm{sign}\rangle \cdot \mathrm{sign}$
we obtain
 \[
\langle T_f P_0g,P_0h\rangle =\mathbb{E}[f]\mathbb{E}[g]\mathbb{E}[h] +
\langle f, \mathrm{sign} \rangle \cdot \langle g, \mathrm{sign}\rangle \cdot \langle h, \mathrm{sign}\rangle.
\]

Next we bound the terms with $d\le \ell$. By Cauchy--Schwarz we obtain 
\[\left|\langle T_f P_d g, P_d h\rangle\right| \le \|T_f P_dg\|_2\|P_dh\|_2 \le \|T_f\|_{W_{=d}}\|g^{=d}\|_2\|h^{=d}\|_2.\]
By biglobalness of $f,g,h$ we can apply 
Theorem \ref{thm:Eigenvalues of operators} (the spectral bound) to $T_f$ 
and Theorem \ref{thm:level-d for biglobal functions} (the level $d$ inequality)
to $g$ and $h$, obtaining
\[
\left|\langle T_f P_dg, P_d h\rangle \right| 
\le ( n^{-1} (Cr^2 \cdot  cn^{1/3}r^{-2})^3 )^{d/2}
=  (cC^3)^{d/2} < 0.001^d,
\]
choosing $c$ sufficiently small given the absolute constant $C$.
Similarly, for $d>\ell$ we apply Theorem \ref{thm:Eigenvalues of operators} to obtain
\[|\langle T_f P_{>\ell} g, P_{>\ell} h \rangle| 
\le \|T_f\|_{W^{>\ell}}\|g\|_2\|h\|_2
\le \|f\|_2 \|g\|_2\|h\|_2 (e\ell/n)^{\ell/2} < 0.001^\ell,\]
as $n>n_0$ is large and $\ell \le \log \|f\|_2 \le cn^{1/3}$.
As $|\langle f*g,h\rangle -  \langle T_f P_0g,P_0h\rangle |$
is bounded by the sum of these error terms,
which is at most $0.01$, the proof is complete.
\end{proof}

\section{Waring} \label{sec:waring}

Here we apply product mixing to the Waring problem in $A_n$. 
We start by stating a deep result of Larsen and Shalev \cite[Theorem 8.1]{larsen2009word}
on word maps in $A_n$.

\begin{thm}[\cite{larsen2009word}]
For any non-trivial words $w_1,\dots,w_k$ and $\eps>0$ there is $N$ so that 
$| w_1(A_n) \cap \dots \cap w_k(A_n)| > n^{-29/9-\eps} |A_n|$ for all $n \ge N$.
\end{thm}

Given this theorem, and the observation that images of word maps are normal sets,
to prove our Waring result (Theorem \ref{thm:waring}) it suffices to show that
polynomially dense normal sets are sufficiently global. In the following result,
we show that this holds even for much sparser normal sets.
To deduce Theorem \ref{thm:waring}, we note that if $A$ has 
density at least $n^{-1/4}$ in $w_1(A_n) \cap \dots \cap w_k(A_n)$
then it is $n^{1/7}$-global by Theorem \ref{thm:normalglobal}, 
and then applying Theorem \ref{thm:product mixing}
to $(A,A,xA^{-1})$ for any $x \in A_n$ we deduce $A^3 = A_n$.

\begin{thm} \label{thm:normalglobal}
Let $n$ be large and $A \sub A_n$ be a normal set with $\mu(A) > e^{-n^{0.01}}$. 
Then $A$ is $n^{0.01}$-global.
\end{thm}

The idea of the proof is that $A$ can be decomposed into conjugacy classes,
each of which is either sufficiently global or insignificant in measure.
To analyse restrictions of conjugacy classes, it will be helpful to observe
that any restriction $A_{I \to J}$ can be simplified without significant increase in measure
to some $A_{I' \to J'}$ such that $J'$ is a permutation of $I'$.
This is encapsulated in the following lemma.

\begin{lem} \label{lem:simplify}
Suppose $A \sub A_n$ is a normal set. 
Let $I \to J$ be a restriction with $I,J \in [n]_d$, $d<n/2$. 
Suppose $I' \to J'$ is obtained from $I \to J$ by deleting some $x \to y$
such that either $x \notin J$ or $y \notin I$.
Then $\mu(A_{I \to J}) \le \frac{n-d}{n-2d} \mu(A_{I' \to J'})$. 
\end{lem}

\begin{proof}
Without loss of generality $y \notin I$.
For any $y' \in [n[ \sm (I \cup J)$,
as $A$ is normal, we have a bijection 
between $A_{I'x \to J'y}$ and $A_{I'x \to Jy'}$,
matching each $\sS$ such that $\sS(I'x)=J'y$
to $\sS' = (y y') \sS (y y')$ with $\sS'(I'x)=J'y'$.

Thus $\mu(A_{I' \to J'}) = \mb{E}_{y'} \mu(A_{I'x \to J'y'})
\ge \frac{n-2d}{n-d} \mu(A_{I \to J})$.
\end{proof}

Next we relate globalness of conjugacy classes to their cycle decomposition.
We write $n_i(C)$ for the number of $i$-cycles in a conjugacy class $C$.

\begin{lem} \label{lem:conjglobal1}
Any conjugacy class $C$ with $n_i(C) \le f^i$ for all $i \in \mb{N}$ is $(2f,n/4)$-global.
\end{lem}

\begin{proof}
We show by induction on $d \le n/4$ that 
$\mu(C_{I \to J}) \le (2f)^d \mu(C)$ for all $I,J$ in $[n]_d$.
If $J$ is not a permutation of $I$ then by Lemma \ref{lem:simplify}
we have $\mu(A_{I \to J}) \le 2 \mu(A_{I' \to J'})$ for some
$I' \to J'$ is obtained from $I \to J$ by deleting some $x \to y$,
so we are done by induction.

If $J$ is a permutation of $I$ then $C_{I \to J}$ is either empty
or obtained by deleting some cycles from $C$.
The required bound now follows by repeatedly applying 
the following consequence of the centraliser theorem: if $C'$ is obtained 
by deleting an $i$-cycle from $C$ then $\mu(C) = in_i(C) \mu(C')$.
Indeed, denoting the number of $i$-cycles thus deleted by $a_i$ for each $i$,
as $\sum ia_i \le d$ we obtain $\mu(C) \le \prod_i (if^i)^{a_i} \le (2f)^d$. 
\end{proof}

Now we estimate the measure of conjugacy classes
according to a dyadic decomposition by globalness.

\begin{lem} \label{lem:conjglobal2}
For $r \in 2\mb{N}$ let $X_r$ be the union of all conjugacy classes in $A_n$
that are $(2r,n/4)$-global but not $(r,n/4)$-global.
Then $\mu(X_r) \le (8/r)^{r/2}$.
\end{lem}

\begin{proof}
For each conjugacy class $C \sub X_r$,
by Lemma \ref{lem:conjglobal1} we have
some $i \in \mb{N}$ with $n_i(C) > (r/2)^i$.
By \cite[Lemma 6.1]{larsen2008characters},
we have $\mb{P}_{\sS \sim S_n}( n_i(\sS) \ge f ) \le 1/(f! i^f)$,
so $\mu(X_r) \le \sum_{i \ge 1} (r/2)^i!^{-1} i^{-(r/2)^i} \le (8/r)^{r/2}$.
\end{proof}

We conclude this section by deducing the globalness of normal sets,
thus completing the proof of our Waring result.

\begin{proof}[Proof of Theorem  \ref{thm:normalglobal}]
Let $A \sub A_n$ be a normal set with $\mu(A) > e^{-n^{0.01}}$. 
Consider any restriction $I \to J$ with $I,J$ in $[n]_d$.
If $d>n^{0.01}$ we have the trivial estimate
$\mu(A_{I \to J}) \le 1 \le e^d \mu(A)$.
Otherwise, writing $k_0 = 0.01\log_2 n - 2$,
we note that $A' := A \sm \bigcup_{k \ge k_0} X_{2^k}$ is $\tfrac12 n^{0.01}$-global, 
so by Lemma  \ref{lem:conjglobal2} we estimate
\begin{align*}
\mu(A_{I \to J}) - (\tfrac12 n^{0.01})^d \mu(A')
 & \le \sum_{k \ge k_0} 2^{k+1} \mu(X_{2^k}) 
 \le  \sum_{k \ge k_0} 2^{k+1} (2^{3-k})^{2^{k-1}}.
\end{align*}
We deduce $\mu(A_{I \to J}) \le n^{0.01d} \mu(A)$, as required.
\end{proof}

\section{Roth} \label{sec:roth}

Now we apply product mixing to prove our variant of Roth's Theorem in $S_n$ (Theorem \ref{thm: Our variant of roth_intro}).
We note that if $A \sub A_n$ has no 3AP then the squaring map ${\sf square} := (x \mapsto x^2)$ is injective on $A$. 
Thus Theorem \ref{thm: Our variant of roth_intro} is an immediate consequence 
of Theorem \ref{thm:product mixing} and the following lemma, to be proved in this section.

\begin{lem}\label{lem:roth}
Let $A \sub A_n$ with $\mu(A) \ge n^{-c\log n}$ and ${\sf square}$ injective on $A$,
 with $c>0$ small and $n \in \mb{N}$ large. Then we can find $B,C \sub A$ and restrictions
$B_{I\to J, J\to K}$, $B_{J\to K, K\to L}$, ${\sf square}(C)_{I\to K,J\to L}$ with $|I| < \sqrt{n}$
that are all $n^{1/8}$-global with density at least $e^{-n^{0.01}}$.
\end{lem}

There will be three steps in the proof.
\begin{itemize}
\item
Step 1 will find $B \sub A$ with some restriction $B_{J \to K}$ that is strongly global,
in that it remains global even under further restrictions of comparable size,
so that $B_{I\to J, J\to K}$, $B_{J\to K, K\to L}$ chosen later will be automatically global.
\item
Step 2 will use the strong globalness from Step 1 and averaging to find
a dense restriction $C_{I\to J, J\to K,K \to L}$ with suitable $I$ and $L$.
\item
Step 3 will show that ${\sf square}(C)_{I\to K,J\to L}$ is global with good density.
By construction this also holds for $B_{I\to J, J\to K}$ and $B_{J\to K, K\to L}$,
so this will complete the proof.
\end{itemize}

\begin{proof}
Let $A \sub A_n$ with $\mu(A) \ge n^{-c\log n}$ and ${\sf square}$ injective on $A$,
 with $c>0$ small and $n \in \mb{N}$ large.

\emph{Step 1.}
We find $B_{J \to K}$ with $B \sub A$ by the following iterative process. 
Let $B^0 = A$ and $J_0=K_0=\es$.
At each stage $i \ge 0$, we check whether 
$B^i_{J_i \to K_i}$ has a restriction $B^i_{J_{i+1} \to K_{i+1}}$ 
with $\mu(B^i_{J_{i+1} \to K_{i+1}}) \ge 2n^{r_i/20} \mu(B^i_{J_i \to K_i}) $
and $|J_{i+1}|\le 3|J_i| + r_i$ for some $r_i \in \mb{N}$.
If not, we terminate with $B=B^i$, $J=J_i$ and $K=K_i$;
if so, we fix such a restriction, and proceed to stage $i+1$
with $B^{i+1}$ obtained from $B^i$ by deleting
(i) all $\sS$ such that $\sS(J_{i+1})=K_{i+1}$ and $\sS^{-1}(J_{i+1})=I$ 
with $\mu(B^i_{I \to J_{i+1}, J_{i+1} \to K_{i+1}}) 
< \tfrac14 \mu(B^i_{J_{i+1}\to K_{i+1}})$, and
(ii) all $\sS$ such that $\sS(J_{i+1})=K_{i+1}$ and $\sS(K_{i+1})=L$ 
with $\mu(B^i_{J_{i+1} \to K_{i+1}, K_{i+1} \to L}) 
< \tfrac14 \mu(B^i_{J_{i+1}\to K_{i+1}})$.
We note that 
$\mu(B^{i+1}_{J_{i+1} \to K_{i+1}})
 \ge  \mu(B^i_{J_{i+1} \to K_{i+1}})
  \ge n^{r_i/20} \mu(B^i_{J_i \to K_i})$.
This process terminates at some $i=i^*$
with $\sum_{i=1}^{i^*} r_i \le 20c\log n$,
and so $|J| \le n^{50c}$.

\emph{Step 2.}
By averaging, we can find $I$ so that
$\mu(B_{I \to J,J \to K}) \ge \mu(B_{J \to K})$.
Let $B'$ be obtained from $B$ by deleting all $\sS$
with $\sS(I)=J$, $\sS(J)=K$ and $\sS(K') \cap I \ne \es$,
identifying tuples with their underlying sets
and writing $K' := K \sm (I \cup J)$, 
By Step 1, $B_{I \to J,J \to K}$ is $n^{1/8}$-global,
so \[\mu(B'_{I \to J, J \to K}) \ge  
\mu(B_{I \to J, J \to K}) - n^{-7/8} |K'||I| \mu(B_{I \to J,J \to K})
\ge \tfrac12 \mu(B_{I \to J, J \to K}).\]
By averaging, we can find $L$ so that
$C := \{ \sS \in B': \sS(I) = J, \sS(J)=K, \sS(K)=L \}$
has $\mu(C_{I\to J,J\to K,K\to L}) \ge \mu(B'_{I \to J, J \to K})$.
Letting $L' \sub L$ correspond to $K'$, clearly $I \cap L' = \es$.

\emph{Step 3.}
Clearly $B_{I \to J, J \to K}$ and $B_{J \to K, K \to L}$
are non-empty, so by Step 1 they are $n^{1/8}$-global with density
at least  $\tfrac14 \mu(B_{J\to K}) \ge \tfrac14 \mu(A) \ge e^{-n^{0.01}}$.
It remains to show that this also holds for $C' := {\sf square}(C)_{I\to K,J\to L}$.
As ${\sf square}$ is injective on $C$ and $|J| \le n^{20c}$ with $c$ small we have
\begin{align*} 
\mu(C')  & = \mu(C_{I\to J, J\to K, K\to L}) \frac{ |U_{I\to J,J\to K,K\to L}| }{ |U_{I\to K, J\to L}| } \\
& \ge \frac{1}{4}\mu(C)n^{-|J|} \ge e^{-n^{0.01}}. 
\end{align*}
Now we consider globalness of $C'$.
Fix any restriction $M \to N$ of $C'$
with $M \sub [n] \sm (I \cup J)$ and $N \sub [n] \sm (K \cup L)$.
We can assume $|M| \le n^{0.01}$, otherwise we are done by the trivial bound
$\mu(C'_{M \to N}) \le 1 \le \mu(C') e^{|M|}$.

As ${\sf square}$ is injective on $C$, for any restriction $M \to N$ of $C'$
with $M \sub [n] \sm (I \cup J)$ and $N \sub [n] \sm (K \cup L)$  we have
\[ |C'_{M \to N}| = \sum_S |C^*_{M \to S, S \to N}|,
\text{ where } C^* := C_{I \to J,J \to K,K \to L}, \text{ so } \]
\begin{equation} \label{eq:C'}
\frac{\mu(C'_{M \to N})}{\mu(C')} = \sum_S  \frac{\mu(C^*_{M \to S, S \to N})}{\mu(C^*)}
   \frac{ |U_{I\to K, J\to L}| }{ |U_{I\to K, J\to L, M \to N}| }
   \frac{ |U_{I\to J,J\to K,K\to L,M \to S, S \to N}| }{ |U_{I\to J,J\to K,K\to L}| }. 
\end{equation}
   
By the strong globalness from Step 1, 
each $\mu(C^*_{M \to S, S \to N}) \le n^{|M|/9} \mu(C^*)$.
We also have
$\frac{ |U_{I\to K, J\to L}| }{ |U_{I\to K, J\to L, M \to N}| } = \frac{(n-|I \cup J|)!}{(n-|I \cup J|-|M|)!}$.

Next we consider any coordinate restriction $x \to y$ in $M \to N$ 
and for each $S$ in the sum above, we let $s \in S$ be such that 
$x \to s$ is in $M \to S$ and $s \to y$ is in $S \to N$. We claim that 
if both $x \in K$ and $y \in J$ then there is no choice of $s$ 
such that $U_{I\to J,J\to K,K\to L,M \to S, S \to N} \ne \es$.
Indeed, as $M \to N$ is a restriction of $C' := {\sf square}(C)_{I\to K,J\to L}$,
we have $M \cap (I \cup J) = \es$ and $N \cap (K \cup L)=\es$,
so $x \to S$ must be in $K' \to L'$ and $s \to y$ in $I \to J$,
but this contradicts $I \cap L'=\es$, so the claim holds.

We can assume this does not occur for any $x \to y$ in $M \to N$,
so we can write $M \to N$ as $M_0 M_1 M_2 \to N_0 N_1 N_2$,
where $M_1 \sub K$, $M_0 M_2 \cap K = \es$,
$N_2 \sub J$ and $N_0 N_1 \cap J = \es$.
Writing $S = S_0 S_1 S_2$ accordingly, we see that
when $U_{I\to J,J\to K,K\to L,M \to S, S \to N} \ne \es$
we have $S_1 \sub L$ determined by $K \to L$ 
and $S_2 \sub I$ determined by $I \to J$.
Thus $U_{I\to J,J\to K,K\to L,M \to S, S \to N}$
is obtained from $U_{I\to J,J\to K,K\to L}$
by imposing the additional restrictions
$M_0 \to S_0$, $S_0 \to N_0$,
$M_2 \to S_2$ and $S_1 \to N_1$.

Considering these additional restrictions,
we note that $M_2 \to S_2$ cannot overlap
$M_0 \to S_0$ as $M_0 \cap M_2 = \es$,
and that $M_0 M_2 \to S_0 S_2$ cannot overlap
$S_1 \to N_1$ as $S_1 \sub K$ is disjoint from $M_0M_2$.
We write $S_0 \to N_0$ as $S'_0 S''_0 \to N'_0 N''_0$,
where $S''_0 \to N''_0$ is contained in $M_0 M_2 S_1 \to S_0 S_2 N_1$
and $S'_0 \to N'_0$ is disjoint from $M_0 M_2 S_1 \to S_0 S_2 N_1$.

Returning to \eqref{eq:C'}, 
we have $\le |M|!$ choices of $S''_0$,
and writing $x = |S'_0| \le |M|$, 
we have $<n^x$ choices of $S'_0$, so
\begin{align*}
\frac{\mu(C'_{M \to N})}{\mu(C')} 
& \le \sum_x |M|! n^x \cdot n^{|M|/9} \mu(C^*) 
  \cdot  \frac{(n-|I \cup J|)!(n-|I \cup J \cup K|-|M|-x)!}{
    (n-|I \cup J|-|M|)!(n-|I \cup J \cup K|)!} \\
& \le n^{|M|/8}  \mu(C^*), 
\end{align*}
as $|K| < n^{50c}$ with $c$ small and $|M| \le n^{0.01}$.
Thus $C'$ is $n^{1/8}$-global.
\end{proof}

\section{Sharpness for mixing}\label{sec:Sharpness}

In this final section we analyse an example that will establish sharpness of our results on mixing.
We start by presenting the example, which is simple to describe, although its analysis will be somewhat involved.

\begin{ex}
Let $n$ be even, $k=n/2$ and $\ell = (1+\rho)n/4 \in [n]$ with $\rho \in (-1,1)$.
Let $A = \{ \sS \in S_n: |\sS([k]) \cap [k]| = \ell\}$ and $f=1_A/\mu(A)$.
\end{ex}

We start by estimating the measure of $A$.

\begin{lem}
We have $n^{-1} \log \mu(A) = - \tfrac12 \rho^2 + O(\rho^4) + O(\tfrac{\log n}{n})$.
\end{lem}

\begin{proof}
We have $\mu(A) = \frac{k!^4}{n! \ell!^2 (k-\ell)!^2}$, so by Stirling's formula
\[n^{-1} \log \mu(A) = 4 \cdot \tfrac12 \log \tfrac12 
- 2 \cdot  \tfrac{1+\rho}{4} \log  \tfrac{1+\rho}{4}
- 2 \cdot  \tfrac{1-\rho}{4} \log  \tfrac{1-\rho}{4}
+ O(\tfrac{\log n}{n}).\]
The lemma follows by considering the Taylor expansion.
\end{proof}

Next we consider globalness of $A$.

\begin{lem}
$A$ is $4$-global.
\end{lem}

\begin{proof}
We consider the covering map from $S_n$ to the slice $\tbinom{[n]}{k}$,
defined by  $\sS \mapsto \sS([k])$, which is a $k!(n-k)!$-fold covering.
Restricting to $A$ gives a $k!(n-k)!$-fold covering of
$\mc{G} := \{ B \in \tbinom{[n]}{k}: |B \cap [k]| = \ell \}$,
so we have $\mu(A) = \mu(\mc{G})$.

To describe the measures of restrictions we introduce some notation:
for any $S, T \sub [n]$ and $\mc{F} \sub \tbinom{[n]}{k}$
we write $\mc{F}_{S,\ov{T}}$ for the set of $B \in \mc{F}$
with $S \sub B$ and $B \cap T = \es$. Now let $I \to J$ be any restriction,
let $S \sub J$ correspond to $I \cap [k]$ and let $T = J \sm S$.
Considering the restriction of the above covering map 
from $U_{I,J}$ to  $\tbinom{[n]}{k}_{S,\ov{T}}$, and the further restriction 
from $A \cap U_{I,J}$ to $\mc{G}_{S,\ov{T}}$, we see that 
\begin{align*}
\mu(A_{I \to J}) & = \mu(\mc{G}_{S,\ov{T}})
\le \frac{|\mc{G}|}{|\tbinom{[n]}{k}_{S,\ov{T}}|}
\le \mu(\mc{G}) \frac{\tbinom{n}{k}}{\tbinom{n-|T|}{k-|S|}}
\le \mu(\mc{A}) 4^{|S|+|T|}. \qedhere
\end{align*}
\end{proof}

\subsection{Product mixing}

Now we consider product mixing. For any $\sS \in A$ we let $N$ be the number of $\tT \in A$
such that $\tau \sS \in A$. Clearly $N$ is independent of $\sS$. By somewhat tedious direct
calculation (we omit the details) one can show
\[ N = \sum_a N_a \ \text { where } 
  N_a = \frac{k!^2\ell!^2(k-\ell)!^2}{a!(\ell-a)!^3(k-2\ell+a)!^3(3\ell-k-a)!}. \]
Here $N_a$ counts those $\tau$ with $|\tau(\sS([k]) \cap [k]) \cap [k]|=a$.

Write $a  = \aA n$ and consider
\[ N_a/N_{a-1} = \frac{(\ell-a+1)^3 (3\ell-k-a+1)}{a(k-2\ell+a)^3} = \bB/\gG + O(1/n), \]
where $\bB = \bB(\aA,\rho) = (1+\rho - 4\aA)^3  ( 1 + 3\rho - 4\aA)$
and $\gG = \gG(\aA,\rho) = 4\aA ( 4\aA - 2\rho )^3$.
We have $\bB - \gG =  (1+3\rho-8\aA) ( (1+\rho)^3 - 8\aA(1+3\rho-4\aA) )$,
from which it is easy to see that $N_a$ is unimodal,
with a maximum at $a$ within $\pm O(1)$
of $\aA n$ where $\aA = \tfrac{1+3\rho}{8}$. 

As $\langle f*f, f \rangle =  \frac{N}{\mu(A)^2 n!}$,
by Taylor expansion we estimate
\[ n^{-1} \log \langle f*f, f \rangle = \rho^3 + O(\rho^4) + O(\tfrac{\log n}{n}). \]
In particular, we see that if $\rho = -n^{-1/3+c}$ with $c>0$
then $\langle f*f, f \rangle = e^{-\TT(n^{3c})}$ 
and $\mu(A) = e^{-\TT(n^{1/3+2c})}$,
so the exponent $1/3$ is optimal in Theorem \ref{thm:The global mixing property_Intro}.

\subsection{Linearisation}

Calculations as above for product mixing
seem prohibitively technical for general mixing,
but fortunately they can be avoided by linearisation.

We recall some theory for linear functions on $S_n$ 
developed by Keevash, Lifshitz and Minzer~\cite{keevash2022largest}. 
As before, we define the dictators $x_{i \to j}$ by  $x_{i\to j}(\sS) = 1_{\sS(i) =j}$.
Any function $\varphi$ that is a linear combination of dictators is called \emph{linear};
if also $\mb{E}\varphi=0$ we say $\varphi$ is \emph{purely linear}. 
Any purely linear function $\varphi$ has a \emph{canonical form},
which is its unique expression $\varphi = \sum_{i,j} a_{ij}x_{i\to j}$ 
satisfying $\sum_i a_{ij} = 0$ for all $j$ and $\sum_{j}a_{ij} = 0$ for all $i$. 
We write $M_{\varphi} = (a_{ij})_{i,j}$ for the matrix of these coefficients.

\begin{lem} $ $ \label{lem:lin}
\begin{enumerate}
\item Let $\varphi = \sum a_{ij}x_{i\to j}$ be in canonical form
and $\varphi' = \sum b_{ij}x_{i\to j}$ be any linear function.
Then $ \langle \varphi, \varphi'\rangle = \tfrac{1}{n-1} \sum_{i,j} a_{ij} b_{ij}$.
(See \cite[Lemma 3.1]{keevash2022largest}).
\item
Let $g\in L^2(S_n)$ and let $(a_{ij}) = M_{g^{=1}}$. 
Then $a_{ij}  = \frac{n-1}{n}(\mb{E}[g_{i\to j}] - \mb{E}[g])$.
(See \cite[Lemma 3.2]{keevash2022largest}).
\item
If $\varphi, \varphi'$ are linear 
then $M_{\varphi*\varphi'} =\tfrac{1}{n-1} M_{\varphi} M_{\varphi'}$.
(See \cite[Lemma 3.4]{keevash2022largest}).
\end{enumerate}
\end{lem}

Next we compute the linearisation $f^{=1}$ of $f=1_A/\mu(A)$ as above.
We write $M = \left(\begin{array}{cc} J & -J\\ -J & J \end{array}\right)$,
where $J$ is the $k \times k$ all-1 matrix.
We normalise by letting $\varphi$ be the linear function 
with $M_\varphi = \frac{\sqrt{n-1}}{n} M$,
so that $\|\varphi\|_2 = 1$ by  Lemma \ref{lem:lin}.1 and 
$\|\varphi\|_4 = O(1)$ by \cite[Lemma 3.5 and Theorem 2.6]{keevash2022largest}.

\begin{lem} \label{lem:f1}
For any $m \in \mb{N}$ we have 
$M_{(f^{*m})^{=1}} = \rho^m \sqrt{n-1} M_\varphi = \rho^m \tfrac{n-1}{n} M$.
\end{lem}

\begin{proof}
We use Lemma \ref{lem:lin}.2, noting that $\mb{E}f=1$
and $\mb{E}[f_{i \to j}] = n\mb{P}_{\sS \sim A} [\sS(i)=j]$,
which equals $ \frac{n\ell}{k^2}$ if $i,j$ are both in $[k]$,
or $\frac{n(n-2k+\ell)}{(n-k)^2}$ if $i,j$ are both not in $[k]$,
or $\frac{n(k-\ell)}{k(n-k)}$ otherwise.
Substituting $k=n/2$ and $\ell = (1+\rho)n/4$ proves the lemma for $m=1$.
General $m$ follows by Lemma \ref{lem:lin}.3, using $M^m = n^{m-1} M$.
\end{proof}

\subsection{Mixing}

Now we analyse the general mixing properties of $f=1_A/\mu(A)$ as above.
The following bound for $L^2$ mixing
shows sharpness of Theorem \ref{thm:mixing time};
the proof is immediate from 
$\| f^{*m} - (f^{*m})^{=0} \|_2 \ge \| (f^{*m})^{=1} \|_2 =  \rho^m \sqrt{n-1} $, 
using the explicit formula for the linearisation computed in Lemma \ref{lem:f1}.

\begin{lem}
If $\rho = \OO(n^{-1/2m})$ with $m$ fixed and $n$ large
then $\|f\|_2 = e^{\OO(n^{1-1/m})}$
and $\| f^{*m} - (f^{*m})^{=0} \|_2 = \OO(1)$.
\end{lem}

For $L^1$ mixing we also need the following estimate for $f^{=2}$.

\begin{lem} \label{lem:f2}
There is an absolute constant $C$ so that for $m \le \tfrac{1}{40} \rho^2 n$
we have $\|(f^{*m})^{=2}\|_2\le (C\rho^2)^m n$. 
\end{lem}

\begin{proof}
Recall that $f$ is $2$-global with $\|f\|_1=1$ and $\|f\|_2 = e^{\Theta(n\rho^2)}$.
By the level $2$ inequality (Theorem \ref{thm:level-d for biglobal functions}) 
we have $\|f^{=2}\|_2\le O(\rho^2 n)$. 
By Theorem  \ref{thm:Eigenvalues of operators} we deduce $\|T_{f}\|_{V_{=2}} \le O(\rho^2)$.
The lemma follows as $\|(f^{*m})^{=2}\|_2 \le \|T_{f}\|_{V_{=2}}^{m-1} \|f^{=2}\|_2$.
\end{proof}

We conclude by showing that  Theorem \ref{thm:mixing time} is also sharp for $L^1$ mixing.

\begin{thm}\label{thm:Sharpness}
For any $\eps>0$ and  $\rho = \OO(n^{-1/2m})$ with $m$ fixed and $n > n(m,\eps)$ large,
we have $\|f\|_2 = e^{\OO(n^{1-1/m})}$ and
the $(L_1,\epsilon)$-mixing time of $f$ is larger than $m$.
\end{thm}

\begin{proof}
We will show that $\varphi$ is a statistic that distinguishes
the probability distribution $\nu_{f^{*m}}$ from the uniform measure $\mu$.
As $\varphi$ is purely linear, we have
$\nu_{f^{*m}}(\varphi) = \langle f^{*m}, \varphi \rangle = \| (f^{*m})^{=1} \|_2 =  \rho^m \sqrt{n-1}$
by Lemma \ref{lem:f1}. Also, by Cauchy--Schwarz and Lemma \ref{lem:f2} we have
\[ \nu_{f^{*m}}(\varphi^2)  =
 \langle f^{*m} , \varphi^2 \rangle = \langle \left( f^{*m}  \right)^{\le 2} , \varphi^2 \rangle
   \le \|\left(f^{*m}\right)^{\le 2}\|_2 \|\varphi^2\|_2 \le (C\rho^2)^m n,\] 
as $\| \varphi \|_4 = O(1)$. By the Paley--Zygmund inequality,  if $\sS \sim \nu_{f^{*m}}$
then $\varphi(\sigma)=\Omega(\rho^{-1})$ with probability $\TT(1)$.
On the other hand, under the uniform measure we have
$\mu(\varphi^2) = \| \varphi \|_2^2 = 1$,
so by Markov's inequality $\varphi(\sigma)=\Omega(\rho^{-1})$ 
with probability $O(\rho) = o(1)$. Thus 
$d_{TV}(\nu_{f^{*m}},\mu) = \frac{1}{2}\| f^{*m} - 1 \|_1 = \TT(1)$.
\end{proof}

\bibliographystyle{plain}
\bibliography{refs}

\end{document}